\documentclass[11pt]{amsart}
\usepackage{latexsym,amsfonts,amssymb,graphicx,amsmath}
\usepackage{epstopdf}
\usepackage{enumerate}
\usepackage{enumitem}
\usepackage{booktabs}
\usepackage{stmaryrd}
\usepackage{sidecap}
\usepackage{color}
\usepackage[font=small,labelfont=bf]{caption}
\usepackage{color,soul}
\newcommand{\Rd}{\color{black}}
\newcommand{\Bk}{\color{black}}
\numberwithin{equation}{section}

\newcommand{\bfz}{{\bf 0}}
\newcommand{\decvar}{{\mathcal D}}

\newcommand{\n}{{\bf n}}
\newcommand{\m}{{\bf m}}
\newcommand{\tv}{{\bf t}}
\newcommand{\vv}{{\bf v}}
\newcommand{\V}{{\bf V}}

\newcommand{\ww}{{\bf w}}
\newcommand{\rr}{{\bf r}}
\newcommand{\vc}{{\boldsymbol \chi}}
\newcommand{\vx}{{\boldsymbol \xi}}
\newcommand{\D}{{\mathsf{D}}}
\newcommand{\A}{{\mathsf{A}}}
\newcommand{\cde}{{\mathsf{E}}}

\newcommand{\sdiv}{{{\nabla\cdot} \,}}
\newcommand{\set}[2]{\left\lbrace #1 \; : \; #2 \right\rbrace}

\newcommand{\Eh}{\mathcal{E}_h}
\newcommand{\Sh}{\mathcal{S}_h}
\newcommand{\Lh}{\mathcal{L}_h}

\newcommand{\ra}[1]{\renewcommand{\arraystretch}{#1}}

\newcommand{\dive}{\operatorname{div}}

\newcommand{\tr}{\mathcal{R}}

\newcommand{\R}{\mathbb{R}}

\newcommand{\sdim}{\mathbb{D}}

\newcommand{\Th}{\mathcal{T}_h}

\newtheorem{theorem}{Theorem}
\newtheorem{lemma}{Lemma}
\newtheorem{proposition}{Proposition}
\newtheorem{definition}{Definition}

\theoremstyle{definition}

\newtheorem{corollary}{Corollary}

\newtheorem{Def}{Definition}[section]

\begin{document}
\title[Inf-Sup stability of cubic Lagrange Stokes elements]{Cubic Lagrange elements satisfying
\\  exact incompressibility }
\date{\today}


\author[J. Guzm\'an]{Johnny Guzm\'an\textsuperscript{\textdagger}}
\address{\textsuperscript{\textdagger} Division of Applied Mathematics, Brown University, Providence, RI 02912, USA}
\email{johnny\_guzman@brown.edu}
\thanks{}

\author[R. Scott]{L. Ridgway Scott\textsuperscript{\textdaggerdbl}}
\address{\textsuperscript{\textdaggerdbl}Departments of Computer Science and Mathematics,
Committee on Computational and Applied Mathematics,
University of Chicago, Chicago IL 60637, USA}
\email{ridg@uchicago.edu}
\thanks{}

\maketitle

\begin{abstract}  
We prove that an analog of the Scott-Vogelius finite elements are inf-sup stable 
on certain nondegenerate meshes for piecewise cubic velocity fields. 
We also characterize the divergence of the velocity space on such meshes.
In addition, we show how such a characterization relates to the dimension
of $C^1$ piecewise quartics on the same mesh.

\end{abstract}
\medskip

\smallskip

\date{}

\section{Introduction}

In 1985 Scott and Vogelius \cite{scott1985norm} (see also \cite{vogelius1983right}) 
presented a family of piecewise polynomial spaces in two dimensions that yield solutions 
to the Stokes equations with velocity approximations that are exactly divergence free.  
The velocity space consists of continuous piecewise polynomials of degree $k\geq 4$,
and the pressure space is taken to be the divergence of the velocity space.
Moreover, they proved stability of the method by establishing that the pair of spaces 
satisfy the  so-called {\it  inf-sup}  condition assuming that the meshes are
quasi-uniform and that the maximum mesh size is sufficiently small.
In a recent paper \cite{guzmanscottdegfour} we gave an alternative proof of the 
inf-sup stability for $k\geq 4$
on more general meshes, assuming only that they are non-degenerate (shape regular).
One key aspect in the proof is to use the stability of the $P^2-P^0$ (or the 
Bernardi-Raugel \cite{bernardi1985analysis}) finite element spaces. 
As a result the proof becomes significantly shorter.
Here we utilize and extend the techniques to the case $k=3$.
The case $k=3$ has been considered earlier \cite{ref:QinThesis}.  

One key concept in this paper is the notion of a {\it local interpolating vertex}. Roughly speaking, this is an interior vertex $z$ such that for every finite element pressure $p$ we can find a discrete velocity $\vv$  in the finite element velocity space such that $\dive \vv(z)=p(z)$ with support in the patch of $z$ and such that $\dive \vv (\sigma)=0$ for all other vertices.  Moreover, we require that $\dive \vv$ has zero mean on each triangle. We then show that if all interior vertices are local interpolating vertices then the inf-sup stability holds; see Theorem \ref{thm1}. We generalize this result to show that if some interior vertices are local interpolating and that there are  acceptable paths from any other vertex to one of the local interpolating vertices then the inf-sup condition holds; see Theorem \ref{mainthm}. In \cite{guzmanscottdegfour} we showed that all interior vertices are local interpolating vertices if piecewise quartic velocities or higher are used. In this article, we show that a generic interior vertex is local interpolating if piecewise cubics are used for the velocity space. In particular, we show that singular vertices and vertices with odd number of triangles touching it are local interpolating vertices. In the case that a non-singular interior vertex has an even number of triangles touching it then we give sufficient conditions for it to be a local interpolating vertex.  Although a generic interior vertex is a local interpolating vertex, there are important meshes were no interior vertex is locally interpolating (e.g. a diagonal mesh).  

It is known that the $C^1$ piecewise quartic space which we denote by $\widehat S_h^4$ is related to the piecewise cubic Lagrange space. Surprisingly, the dimension of  $\widehat S_h^4$ has not been verified for all meshes, but it has been verified for a large class of meshes; see \cite{lrsBIBbv,alfeld1987explicit}.  In the last sections of this paper we use the onto-ness of the divergence operator of piecewise cubics to verify the dimension of $\widehat S_h^4$ for a large class of meshes. We also are able to verify the dimension of $S_h^4$ which are $C^1$ piecewise quartic space with functions vanishing on the boundary to second-order.

The paper is organized as follows.
In the following section we begin with Preliminaries. 
In Section \ref{sec:locintrpvv} we identify vertices at which we can interpolate 
pressure vertex values with the divergence of localized velocity fields.
In Section \ref{sec:kaythree} we prove the inf-sup 
stability for $k = 3$ under some restrictive assumptions on the mesh. 
In Section \ref{sec:gencase} we characterize the divergence of piecewise cubics 
on a broader class of meshes than considered in Section \ref{sec:kaythree}.
In Section \ref{sec:relateqin} we compare our the results with those 
of \cite{ref:QinThesis}.
In Section \ref{sec:strangdim} we relate our results to the dimension of 
$C^1$ piecewise quartics.

\section{Preliminaries}\label{preliminaries}

We assume $\Omega$ is a polygonal domain in two dimensions. 
We let $\{ \mathcal{T}_h\}_h$ be a non-degenerate (shape regular) family triangulation 
of $\Omega$; see \cite{brenner2007mathematical}. 
The set of vertices and the set of internal edges are denoted by
\begin{alignat*}{1}
\Sh&=\{ x:  x  \text{ is a vertex of } \Th \}, \\
\Eh&=\{ e: e \text{ is an edge of }  \Th \text{ and } e \not\subset \partial \Omega \}.\\ 
 \end{alignat*}
We also let $\Sh^{\text{int}}$ and $\Sh^{\partial}$ denote all the interior vertices and boundary vertices, respectively.

Define the internal edges $\Eh(z)$ and triangles $\Th(z)$ that have $z \in \Sh$ as a vertex via
\begin{equation*}
\Eh(z)=\{ e \in \Eh: z \text{ is a vertex of } e \}, \qquad
\Th(z)=\{ T \in \Th: z \text{ is a vertex of } T \}.  
\end{equation*}
Finally, we define the patch
\begin{equation*}
\Omega_h(z)= \bigcup_{T \in \Th(z)} T .
\end{equation*}
The diameter of this patch is denoted by 
\begin{equation}\label{eqn:aitchzeedef}
h_z= \text{ diam } (\Omega_h(z)).
\end{equation}

To define the pressure space we must define singular and non-singular vertices. 
Let $z \in \Sh$  and suppose that $\Th(z)=\{ T_1, T_2, \ldots T_N \}$,
enumerated so that $T_j, T_{j+1}$ share an edge for $j=1, \ldots N-1$ and  
$T_N$ and $T_1$ share an edge in the case $z$ is an interior vertex. 
If $z$ is a boundary vertex then we enumerate
the triangles such that $T_1$ and  $T_N$ have a boundary edge.
Let $\theta_j $ denote the angle between the edges of $T_j$ originating from $z$. Define
 \[
\Theta(z)= 
\begin{cases}
\max \{ |\sin(\theta_1+\theta_{2})|,  \ldots, |\sin(\theta_{N-1}+\theta_{N})|, |\sin(\theta_N+\theta_1)| \}& \text{ if }  z \text{ is an interior vertex } \\
\max \{ |\sin(\theta_1+\theta_{2})|,  \ldots, |\sin(\theta_{N-1}+\theta_{N})| \} 
     & \text{ if }  z \text{ is a boundary vertex }.
\end{cases}
\]

\begin{Def}
A vertex $z \in \Sh$ is a singular vertex if $\Theta(z)=0$. 
It is non-singular if $\Theta(z) >0$.
\end{Def}
This is equivalent to the original definition given in 
\cite{lrsBIBaf,vogelius1983right}. 

\begin{figure}
\centerline{\qquad\includegraphics[scale=.8]{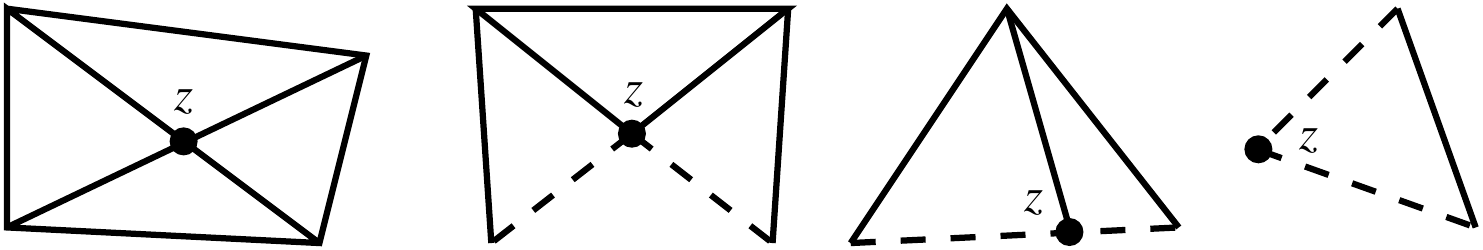}}
\footnotesize
\ra{1.1}
\caption{Example of singular vertices $z$. Dashed edges denote boundary edges. }
\label{fig:singverts}
\end{figure}

We denote all the non-singular vertices by
\begin{equation*}
\Sh^1=\{ x \in \Sh: x \text{ is non-singular, that is, } \Theta(x)>0\},
\end{equation*}
and all singular vertices by $\Sh^2=\Sh \backslash \Sh^1$.  
We also define, for $i=1,2$, $\Sh^{i, \text{int}}=\Sh^i \cap \Sh^{\text{int}}$, 
and  $\Sh^{i, \partial}=\Sh^i \cap \Sh^{\partial}$.

Let $q$ be a function such that $q|_T \in C(\overline{T})$ for all $T \in \Th$. 
For each vertex $z \in \Sh^2$ define
\begin{equation}\label{eqn:ehaitchzeeq}
A_h^z(q)= \sum_{j=1}^N (-1)^{j-1} q|_{T_j}(z).
\end{equation}
The Scott-Vogelius finite element spaces for $k \ge 1$ are defined by
\begin{alignat*}{1}
\V_h^k&=\{v \in [C_0(\Omega)]^2: v|_T \in [P^k(T)]^2, \forall T \in \Th\}, \\
Q_h^{k-1}&= \{ q \in L_0^2(\Omega): q|_T \in P^{k-1}(T), \forall T \in \Th, A_h^z(q)=0 \text{ for all }  z  \in \Sh^2\}. 
\end{alignat*}
Here $P^k(T)$ is the space of polynomials of degree less than or equal to $k$ defined on $T$. Also, $L_0^2(\Omega)$ denotes the subspace of $L^2$ of functions that have average zero on $\Omega$.

We also make the following definition
\begin{equation*}
\V_{h,0}^k=\{\vv\in\V_h^k: \int_T \dive\vv \, dx=0, \quad \text{ for all } T \in \Th \}.
\end{equation*}

The definition of $Q_h^{k-1}$ is based on the following 
result \cite{scott1984conforming}.
\begin{lemma}\label{lemma1}
For $k \ge 1$, 
$\dive \V_h^k \subset Q_h^{k-1}$.
\end{lemma}

The goal of this article is to prove the inf-sup stability of the pair 
$\V_h^k, Q_h^{k-1}$ for $ k = 3$, for certain meshes.
\begin{Def}
The pair of spaces $\V_h^k, Q_h^{k-1}$ is inf-sup stable on a family of 
triangulations $\{\Th\}_h$ if there exists $\beta>0 $ such that for all $h$
\begin{equation}\label{inf-sup}
\beta \, \|q\|_{L^2(\Omega)} \le \sup_{ \vv \in \V_h^{k},  \vv \not\equiv 0}  
\frac{\int_{\Omega} q \dive \vv  \, dx}{ \|\vv\|_{H^1(\Omega)}}  
\quad \quad \forall q \in  Q_h^{k-1}. 
\end{equation}
\end{Def}

\subsection{Notation for piecewise linears}

For every $z \in \Sh$ the function $\psi_z$  is the continuous, piecewise linear 
function corresponding to the vertex $z$. That is, for every $y \in \Sh$
\begin{equation}\label{eqn:weneedthis}
\psi_z(y)= 
\begin{cases}
1 \quad \text{ if } y=z,\\
0 \quad \text{ if } y \not=z. 
\end{cases}
\end{equation}

\begin{figure}
\centerline{\qquad\includegraphics[scale=.7]{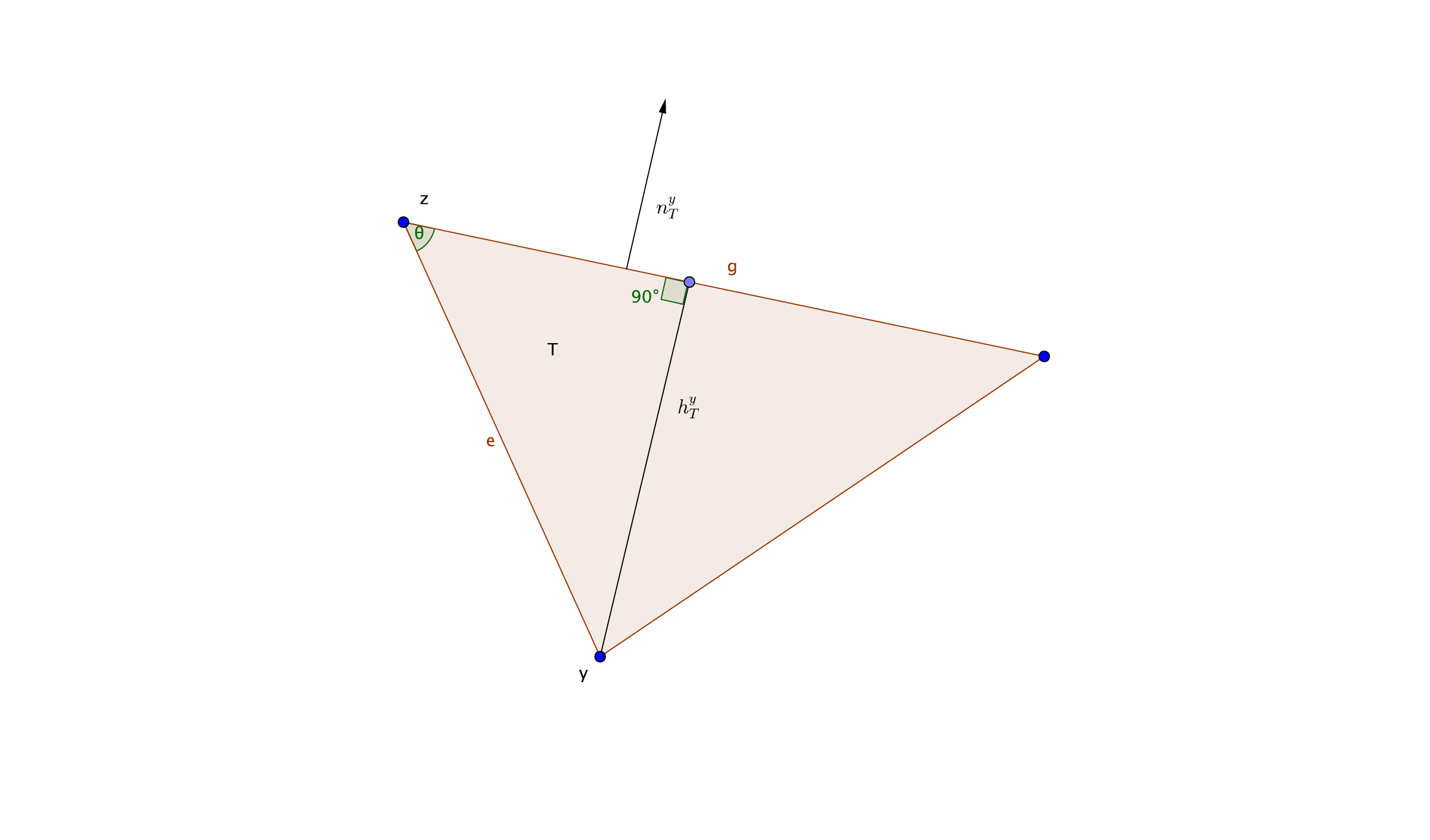}}
\footnotesize
\ra{1.1}
\vspace{-40pt}
\caption{Notation for single triangle}
\label{fig:nT}
\end{figure}

 For $z \in \Sh$ and $e \in \Eh(z)$, where $e=\{ z, y\}$, let $\tv_e^z=|e|^{-1}(y-z)$
be a unit vector tangent to $e$, where $|e|$ denotes the length of the edge $e$.  Then
\begin{equation}\label{tpsi}
\tv_e^z \cdot \nabla \psi_y= \frac{1}{|e|}  \text{ on }  e. 
\end{equation}
More generally, suppose that $T \in \Th(y)$ and let $g$ be the edge of $T$ 
that is opposite to $y$; see Figure \ref{fig:nT}. 
Let $\n_T^y$ be the unit normal vector to $g$ that points out of $T$. Then
\begin{equation}\label{nablapsi}
\nabla \psi_y|_T = - \frac{1}{h_T^y} \n_T^y,   
\end{equation}
where $h_T^y$  is the distance of $y$ to the line defined by the edge $g$. 
If $z$ is another vertex of $T$ and  $e=\{ z, y\}$  then  
\begin{equation}\label{hTj}
h_T^y=  |e|\sin\theta \, ,
\end{equation}
where $\theta$ is the angle between the edges of $T$ emanating from $z$.

\subsection{Preliminary stability results}

The following is a consequence of the stability of the Bernardi-Raugel 
\cite{bernardi1985analysis} finite element spaces.

\begin{proposition}\label{prop1}
Let $k \ge 1$. There exists a constant $\alpha_1$ such that for every $p \in Q_h^{k-1}$  
there exists a $\vv \in \V_h^2$ such that 
\begin{equation*}
\int_{T} \dive \vv \, dx =\int_T p \, dx \quad \text{ for all } T \in \Th,
\end{equation*}  
and 
\begin{equation*}
\|\vv\|_{H^1(\Omega)} \le \alpha_1  \|p\|_{L^2(\Omega)}. 
\end{equation*}
The constant $\alpha_1$ is independent of $p$ and only depends on the shape regularity of the mesh and $\Omega$.
\end{proposition}

The next result is a simple consequence of Lemma 2.5 of \cite{vogelius1983right}
and is based on a simple counting argument.

\begin{proposition}\label{lemma3}
Let $k \ge 1$. There exists a constant $\alpha_2>0$ such that for every $p \in Q_h^{k-1}$ with $p(z) =0$ for all $z \in \Sh$ and $\int_T p \, dx =0$ for all $T \in \Th$ there exists $\vv \in \V_h^k$ such that 
\begin{equation*}
\dive \vv = p\text{ on }  \Omega,
\end{equation*}
and
\begin{equation*}
\|\vv\|_{H^1(\Omega)} \le \alpha_2 \|p\|_{L^2(\Omega)}.
\end{equation*}
\end{proposition}

Using the above results we can prove inf-sup stability as long as we can interpolate pressure vertex values with the divergence of velocity fields.  This is the subject of the next result.

\begin{lemma}\label{lemma4}
Suppose there exists a constant $\alpha_3>0$ such that  for every $p \in Q_h^{2}$ there exists a $\vv \in  \V_{h,0}^3$ satisfying
\begin{alignat}{1}
(\dive \vv-p)(\sigma)&=0 \quad \text{ for all } \sigma \in \Sh,
\end{alignat}
with the bound
\begin{equation*}
\|\vv\|_{H^1(\Omega)}  \le \alpha_3 \|p\|_{L^2(\Omega)}.
\end{equation*}
Then, \eqref{inf-sup} holds for $k=3$ with $\beta =\frac{1}{\alpha_1+ \alpha_3(1+\alpha_1)+\alpha_2(1+ \alpha_3(1+\alpha_1))}.$
\end{lemma}
\begin{proof}
Let $p$ be an arbitrary function in $Q_h^2$. First, we let $\vv_1$ be from Proposition \ref{prop1}. We define $p_1=p-\dive \vv_1$. Then, from our hypothesis let $\vv_2 \in V_{h,0}^3$ be such that $(\dive(\vv_2)-p_1)(\sigma)=0$   for all  $\sigma \in \Sh$. Letting $p_2=p_1-\dive \vv_2$ we see that $p_2$ satisfies the hypotheses of Proposition \ref{lemma3} and we let $\vv_3$ be the resulting vector field. Then, we set $\ww=\vv_1+\vv_2+ \vv_3$. Then we see that 
\begin{equation*}
\dive \ww=p \quad \text{ on } \Omega, \\
\end{equation*}
and
\begin{equation}
\|\ww\|_{H^1(\Omega)} \le  
(\alpha_1+ \alpha_3(1+\alpha_1)+\alpha_2(1+ \alpha_3(1+\alpha_1))) 
\|p\|_{L^2(\Omega)} .\label{ww101}
\end{equation}
Hence,
\begin{equation*}
\|p\|_{L^2(\Omega)}^2 = \int_{\Omega} p \, \dive \ww   dx \le \|\ww \|_{H^1(\Omega)} \sup_{ \vv \in \V_h^{3},  \vv \not\equiv 0}  \frac{\int_{\Omega} p \dive \vv  \, dx}{ \|\vv\|_{H^1(\Omega)}} .
\end{equation*}
The result now follows after applying \eqref{ww101}.
\end{proof}

Hence, one way of proving inf-sup stability is two establish the hypothesis of Lemma \ref{lemma4}. This is going to be the task of the next sections.

\section{Locally interpolating vertex values}
\label{sec:locintrpvv}

In this section we will identify vertices at which we can interpolate pressure vertex 
values with the divergence of velocity fields. We first define the local spaces
\begin{alignat*}{1}
\V_h^k(z)&= \{\vv \in \V_h^k: \text{ supp } \vv \subset \Omega_h(z) \},  \\
\V_{h,0}^k(z)&= \{v \in \V_h^k(z): \int_{T } \dive v \,dx=0, \text{ for all } T \in \Th(z)  \},  \\
\V_{h,00}^k(z)&= \{v \in \V_{h,0}^k(z): \dive v (\sigma)=0, \text{ for all } \sigma \in \Sh, \sigma \neq z \}.
\end{alignat*}
Suppose that $z \in \Sh$ and that $\Th(z)=\{ T_1, T_2,  \ldots, T_N  \}$, ordered
as described following \eqref{eqn:aitchzeedef}.
Then in view of \eqref{eqn:ehaitchzeeq} we define 
\begin{equation*}
W(z)=\{ (a_1, \ldots, a_N) \in \R^N: \text{ if } z \in \Sh^2, \text{ then }  \sum_{j=1}^N (-1)^{j-1} a_j=0\}.
\end{equation*}
Note that if $z \in \Sh^1$ is  non-singular then $W(z)=\R^N$ and there is a constraint if $z$ is singular. 

\begin{definition}\label{localinterpolate}
Let $z \in \Sh$ and suppose that $\Th(z)=\{ T_1, T_2, \ldots, T_N\}$. We say that $z$ is a local interpolating vertex if for every   $(a_1, \ldots, a_N)  \in W(z)$ there exists  a $\vv \in V_{h, 00}^3(z)$ such that 
\begin{equation}\label{aux134}
\dive \vv |_{T_j} (z) = a_j \quad \text{ for all } 1 \le j \le N. 
\end{equation}
If $z \in \Sh$ is a local interpolating vertex then,  given $a=(a_1, \ldots, a_N)  \in W(z)$ we define $M_a=\{ \vv \in V_{h, 00}^3(z): \vv \text{ satisfying } \eqref{aux134}\}$. Also, we set 
\begin{equation}\label{aux123}
D_z= \max_{(a_1, \ldots, a_N) \in W(z) } \min_{\vv \in M_a}  \frac{\|\nabla \vv\|_{L^\infty(\Omega_h(z))}}{\max_{1 \le j \le N} |a_j|} .
\end{equation}
\end{definition}
We let $\Lh$ be the collection of all local interpolating vertices.  Then Definition \ref{aux123}  says that if $z \in \Lh$ then given $a \in W(z)$ there exists  $ \vv \in V_{h, 00}^3(z)$ satisfying  \eqref{aux134} and
\begin{equation*}
\|\nabla \vv\|_{L^\infty(\Omega_h(z))} \le D_z \max_{1 \le j \le N} |a_j|.
\end{equation*}

In the next section we  identify local interpolating vertices. It will be useful to define fundamental vector fields used in \cite{guzmanscottdegfour}. For every $z \in \Sh$  and $e \in \Eh(z)$ with $e=\{ z, y\}$ define
\begin{equation}\label{eqn:sighdef}
\eta_e^z=  \psi_z^2 \psi_y .
\end{equation}
Let $T_1$ and $T_2$ be the two triangles that have $e$ as an edge.
Then we can easily verify the following (see also \cite{guzmanscottdegfour}):
\begin{subequations}\label{eta}
\begin{alignat}{1}
& \text{ supp } \eta_e^z  \subset T_1 \cup T_2 ,   \label{eta1} \\
& \nabla \eta_e^z (\sigma )=0 \quad \text{ for  } 
  \sigma \in \Sh  \text{ and }   \sigma \neq z. \label{eta2} \\
  & \|\nabla \eta_e^z\|_{L^\infty(T_1 \cup T_2)} \le \frac{C}{|e|}  \label{eta3}.
\end{alignat}
\end{subequations}
We then define the vector fields  
\begin{equation}\label{wei}
\ww_e^z=|e| \tv_e^z \eta_e^z.
\end{equation}
The following lemma is proved in \cite{guzmanscottdegfour}.
\begin{lemma}
Let $z \in \Sh$  and $e \in \Eh(z)$ with $e=\{ z, y\}$ and denote the two triangles 
that have $e$ as an  edge as $T_1$ and $T_2$. 
Let $\ww_e^z$ be given by \eqref{wei}. Then
\begin{subequations}\label{w}
\begin{alignat}{1}
& \ww_e^z \in \V_{h,00}^3(z),   \label{w1} \\
& \text{ supp } \ww_e^z \subset T_1 \cup T_2, \text{ and }\dive \ww_{e}^z|_{T_s} (z)= 1 \quad \text{ for } s=1,2 \label{w6}\\
& \|\nabla \ww_e^z\|_{L^\infty(T_1 \cup T_2 )} \le  C. \label{w7} 
\end{alignat}
\end{subequations}
The constant $C$ only depends on the shape regularity of $T_1$ and $T_2$.
\end{lemma}

\subsection{Schematic for interior vertex}\label{scheme}

It will be useful to use the following notation for an interior vertex 
$z \in \Sh^{\text{int}}$. 
We assume that $\Th(z)= \{ T_1, T_2, \ldots, T_N\}$,
enumerated as before so that $T_j, T_{j+1}$ share an edge for $j=0, \ldots N-1$ 
and we identify $T_0$ as $T_N$ (indices modulo $N$).
For $1 \le i \le  N$ we let $e_i$ be the edge shared by $T_i$ and $T_{i+1}$ and $e_N$ is the edge shared by $T_N$ and $T_{1}$. 
We let $y_1, y_2, \ldots, y_N $ be the vertices  such that $e_i=\{z, y_i\}$. 
We  set $y_0=z$. 
Also, $\n_i$ is unit-normal to $e_i$ pointing out of $T_i$ and $\theta_i$ is 
the angle formed by the two edges of $T_i$ emanating from $z$.  
Moreover, let $\tv_i$ be the unit-tangent vector to $e_i$ pointing towards $y_i$.  
The edge opposite to  $z$ belonging to  $T_i$ is denoted by $f_i$. 
The unit-normal to $f_i$ pointing out of $T_i$ is denoted $\m_i$. 
\begin{equation}\label{eqn:emmeye}
\m_i=|f_i|^{-1} (y_i-y_{i-1})^\perp=|y_i-y_{i-1}|^{-1} (y_i-y_{i-1})^\perp.
\end{equation}
Finally, $h_i=\text{dist}(z, f_i)$. The notation $\cdot^\perp$  denotes rotation by $90$ degrees counter clockwise.  See Figure \ref{fig:scheme} for an illustration. 

Using this notation, we will use the shorthand notation $\psi_i= \psi_{y_i}$ for $0 \le i \le N$ and $\ww_i= \ww_{e_i}^z$ for $1 \le i \le N$.

\begin{figure}
\centerline{\qquad\includegraphics[scale=.8]{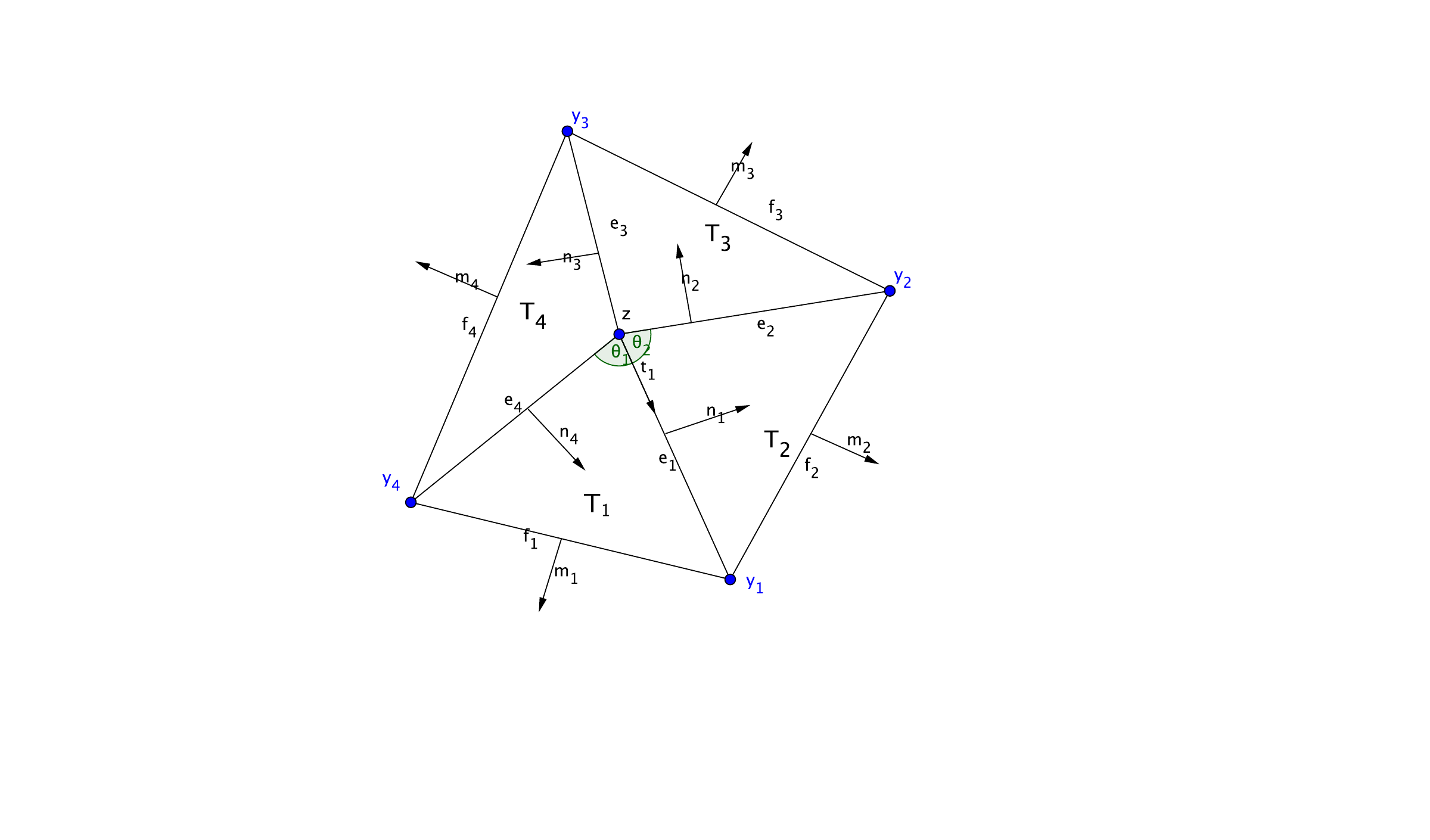}}
\footnotesize
\ra{1.1}
\vspace{-40pt}
\caption{Example of $\Th(z)$ with $N=4$. }
\label{fig:scheme}
\end{figure}

\subsection{ Singular vertices are local interpolating vertices: $\Sh^2 \subset \Lh$}
In this section we recall that all singular vertices are local interpolating vertices. The proof can be found in \cite{guzmanscottdegfour}, but we recall some of the details. 

\begin{lemma}\label{lemma0}
It holds $\Sh^2 \subset \Lh$.  Moreover, there exists a constant $C_{\text{sing}}$ such that
\begin{equation*}
D_z \le C_{\text{sing}} \qquad \text{ for all } z \in \Sh^2,
\end{equation*}
where $C_{\text{sing}}$ only depends on the shape regularity of the mesh. 
\end{lemma}

\begin{proof}

We only consider interior singular vertices for simplicity (the proof for boundary singular vertices is similar).  Suppose that $z \in \Sh^{2,\text{int}}$ and  we use the notation in Section \ref{scheme}. Note that $N=4$. Let $a=(a_1, \ldots, a_4) \in W(z)$. First define $b_1=a_1$ and inductively define
\begin{equation*}
b_{j}=a_j-b_{j-1} \qquad \text{ for } j=2,3. 
\end{equation*}
Then define
\begin{equation*}
\vv= b_1 \ww_1+ b_2 \ww_2+ b_3 \ww_3.
\end{equation*}
By \eqref{w1}, $\vv \in \V_{h,00}^3(z)$. 
Using \eqref{w6} we see that 
\begin{equation*}
\dive \vv|_{T_j}(z)=a_j \text{ for } 1 \le j \le 3.
\end{equation*}
We also have 
\begin{equation*}
\dive \vv|_{T_4} (z)= b_3=a_3-a_2+a_1=a_4,
\end{equation*}
where we used that $(a_1, a_2, a_3, a_4)  \in W(z)$.
Moreover, using \eqref{w7} we have
\begin{equation*}
\|\nabla \vv\|_{L^\infty(\Omega_h(z))} \le C \, (b_1+b_2+b_3) \le C \max_{1 \le j \le 4} |a_j|,
\end{equation*}
where the constant $C$ only depends on the shape regularity constant. 
\end{proof}

\subsection{Interior vertices with odd number of triangles are local interpolating vertices}

In this section we prove that if $z \in \Sh^{\text{int}}$ and  $\Th(z) $ has 
an odd number of triangles then $z \in \Lh$. 
\begin{lemma} \label{odd}
Let $z \in \Sh^{\text{int}}$ with $\Th(z)=\{T_1, \ldots, T_N\}$ and suppose that $N$ is 
odd.
Then $z \in \Lh$. 
Moreover, there exists a constant $C_{\text{odd}}$ such that 
\begin{equation*}
D_z \le C_{\text{odd}}.
\end{equation*}
Here $C_{\text{odd}}$ is a fixed constant that only depends on the shape regularity of the mesh. 
\end{lemma}

\begin{proof}
We use the notation in Section \ref{scheme}.  We start by defining some auxiliary functions. First, define
\begin{equation*}
\vv_1= \frac{1}{2} \sum_{j=1}^N (-1)^{j-1} \ww_j. 
\end{equation*}
We see that $\vv_1 \in \V_{h,00}^3(z)$, by \eqref{w1}. 
Moreover, using \eqref{w6} we see that 
\begin{equation*}
\dive \vv_1|_{T_j}(z)= \delta_{1j} \text{ for } 1 \le j \le N,
\end{equation*}
where $\delta_{ij}$ is the Kronecker delta function.  
Note that here we used crucially that $N$ is odd. Moreover, we have by \eqref{w7} that
\begin{equation*}
\|\nabla \vv_1\|_{L^\infty(\Omega_h(z))} \le  C,
\end{equation*}
where $C$ only depends on the shape regularity.
Next, we define inductively
\begin{equation*}
\vv_i=\ww_i-\vv_{i-1} \quad \text{ for } 2 \le i \le N.
\end{equation*}
Then, we easily see that $\vv_i \in \V_{h, 00}^3(z)$ and 
\begin{equation}\label{aux31}
\dive \vv_i|_{T_j}(z)= \delta_{ij} \quad \text{ for } 1 \le i,j \le N, 
\end{equation}
and, furthermore, 
\begin{equation}\label{aux32}
\|\nabla \vv_i\|_{L^\infty(\Omega_h(z))} \le  C \quad \text{ for } 1 \le i,j \le N, 
\end{equation}
where $C$ only depends on the shape regularity. 
Now given $a=(a_1, \ldots, a_N) \in W(z)$ we set 
\begin{equation*}
\vv=\sum_{j=1}^N a_j \vv_j.
\end{equation*}
Then, using \eqref{aux31} we get \eqref{aux134}. 
Moreover, using  \eqref{aux32} we get
\begin{equation*}
\|\nabla \vv\|_{L^\infty(\Omega_h(z))} \le  C_{\text{odd}} \max_{ 1 \le  j \le N} |a_j|,
\end{equation*}
where  $C_{\text{odd}}$ is a fixed constant only depending on the shape 
regularity of the mesh. 
\end{proof} 

\subsection{ Interior vertices with even number of triangles}

If $z \in \Sh^{\text{int}}$ is non-singular and has an even number of triangles 
containing it, then it is not necessarily the case that $z \in \Lh$. 
In this section we give sufficient conditions for $z \in \Lh$.  
We use the notation in Section \ref{scheme}. 
To do this, in addition to the vector fields $\ww_i$, we will need other vector fields.  
We start with
\begin{equation}\label{eqn:startwith}
\vc_i=\frac{12}{|e_i|} \eta_{e_i}^z \n_i 
     =\frac{12}{|e_i|} \psi_z^2 \psi_{y_i} \n_i 
\quad \text{ for } 1 \le i \le N.
\end{equation}
In the following lemma, the indices are calculated modulo $N$.
\begin{lemma}
It holds, for $1 \le i \le N$
\begin{subequations}
\begin{alignat}{1}
\text{supp} (\vc_i)  \subset & T_{i} \cup  T_{i+1}, \label{chi1}  \\
(\dive \vc_i)(y_j)&=0 \quad \text{ for } 1 \le j \le N,  \label{chi2}\\
(\dive \vc_i)|_{T_{i}} (z)= \frac{12 \cot(\theta_i)}{|e_i|^2}, & \quad  
(\dive \vc_i)|_{T_{i+1}} (z)= \frac{-12\cot(\theta_{i+1})}{|e_i|^2}, \label{chi3} \\
\int_{T_i} \dive \vc_i \, dx= 1, & \quad 
\int_{T_{i+1}} \dive \vc_i \, dx= -1, \label{chi4} \\
\|\nabla \vc_i\|_{L^\infty(T_i \cup T_{i+1})} \le & \frac{C}{|e_i|^2} . \label{chi5}
\end{alignat}
\end{subequations}
\end{lemma}

\begin{proof}
It follows from the definition \eqref{eqn:sighdef} of $\eta_{e_i}^z$ 
that \eqref{chi1} and \eqref{chi2} hold.
A simple calculation using \eqref{nablapsi} and \eqref{hTj} shows that
$$
\nabla\psi_{y_i}|_{T_i}=\frac{1}{\sin(\theta_i) |e_i|}\n_{i-1}.
$$
Thus
\begin{equation*}
(\dive \vc_i)|_{T_{i}}(z)=\frac{12}{|e_i|} \psi_z^2(z) \nabla\psi_{y_i}|_{T_i}\cdot\n_i
= \frac{12}{\sin(\theta_i) |e_i|^2} \n_{i-1} \cdot \n_{i}
= \frac{12\cot(\theta_i) }{|e_i|^2} .
\end{equation*}
Similarly, we can show that 
\begin{equation*}
(\dive \vc_i)|_{T_{i+1}} (z)= \frac{-12\cot(\theta_{i+1})}{|e_i|^2}.
\end{equation*}
To show \eqref{chi4} we use integration by parts and use that $\eta_{e_i}^z$ 
vanishes on $\partial T_{i} \backslash e_i$ to get
\begin{equation*}
\int_{T_i} \dive \vc_i  \,dx= 12 \int_{e_i} \eta_{e_i} \,ds= 1. 
\end{equation*}
Similarly, we can show that 
\begin{equation*}
\int_{T_{i+1}} \dive \vc_i  \,dx= -1. 
\end{equation*}
To prove \eqref{chi5} we use the definition of $\vc_i$ and the bound \eqref{eta3}.
\end{proof}

Note that $\vc_i$ is not in $\V_{h,0}^3(z)$ by \eqref{chi4}. 
However, again using  \eqref{chi4}, we see that 
\begin{equation*}
\vc:=\vc_1+ \vc_2+ \cdots + \vc_N,
\end{equation*}
does belong to  $\V_{h,0}^3(z)$. 
In fact, using \eqref{chi2} we have that $\vc \in \V_{h,00}^3(z)$. 
We collect it in the following result.
\begin{lemma}
It holds that $\vc \in \V_{h,00}^3(z)$ and
\begin{alignat}{1}
(\dive \vc)|_{T_j} (z)
&=12 \cot(\theta_j) \Big(\frac{1}{|e_j|^2}-\frac{1}{|e_{j-1}|^2}\Big)  
\quad \text{ for } 1 \le j \le N, \label{chiz} \\
\|\nabla \vc\|_{L^\infty(\Omega_h(z))} & \le  \frac{C}{h_z^2}.  \label{chiz2}
\end{alignat}
\end{lemma}
The inequality \eqref{chiz2} follows from \eqref{chi5}.

So far, we have $\ww_1, \ldots, \ww_N$ and $\vc$ that belong to $\V_{h,00}^3(z)$. 
Next we describe two more functions that also belong to the space. 

We let $\cde_1=[1,0]^t$ and $\cde_2=[0,1]^t$ be canonical directions. 
We then define
\begin{equation*}
\tilde{\vx}_{i}:=\psi_z^2 \cde_i \quad \text{ for } i=1,2.
\end{equation*}
The following result can easily be proven.

\begin{lemma}
It holds, for $i=1,2$ 
\begin{subequations}\label{xi}
\begin{alignat}{1}
\tilde{\vx}_i \in & \V_h^3(z)  \label{xi0} \\
(\dive \tilde{\vx}_i)(y_j)&= 0\quad \text{ for } 1 \le  j \le N  \label{xi1} \\\
(\dive \tilde{\vx}_i)|_{T_j} (z)&=\frac{3}{|T_j|} b_{ji} \quad 
        \text{ for } 1 \le j \le N \label{xi2} \\
\int_{T_j} \dive \tilde{\vx}_i \, dx&= b_{ji} \quad\text{ for } 1\le j\le N, \label{xi3}
\end{alignat}
\end{subequations}
where $|T_j|$ denotes the area of $T_j$, and using \eqref{eqn:emmeye} we get \Rd
\begin{equation}\label{eqn:beejayeye}
b_{ji}=\frac{-|f_j|\m_j \cdot \cde_i }{3} 
      =\frac{-(y_{j}- y_{j-1})^\perp  \cdot \cde_i }{3}
      =\frac{-(y_{j}- y_{j-1}) \cdot \cde_i^\perp  }{3},
\end{equation}
\Bk
where $\vv^\perp$ denotes the rotation of $\vv$ by 90 degrees counter clockwise. Moreover, the following bound holds
\begin{equation}
\|\nabla \tilde{\vx}_i \|_{L^\infty(\Omega_h(z))}  \le  \frac{C}{h_z}.  \label{xi4}
\end{equation}
\end{lemma}

\begin{proof}
From the definition $\psi_z$ we have that \eqref{xi0} and \eqref{xi1} hold.
To show \eqref{xi2} we have 
\begin{equation*}
(\dive \tilde{\vx}_i)|_{T_j} (z)= 2 \psi_z(z) \nabla \psi_z|_{T_j} \cdot \cde_i
=\frac{-2 \m_j \cdot \cde_i}{h_j}= \frac{- |f_j| \m_j \cdot \cde_i}{ |T_j|}.
\end{equation*}
In the last equation we used that $|T_j|=\frac{h_j |f_j|}{2}$.
Using that $\dive \tilde{\vx}_i$ is linear in $T_j$ and  \eqref{xi1}, \eqref{xi2} we have
\begin{equation*}
\int_{T_j}\dive\tilde{\vx}_i\,dx=\frac{|T_j|}{3}(\dive\tilde{\vx}_i)|_{T_j}(z)=b_{ji}.
\end{equation*}
Finally, \eqref{xi4} follows by a simple computation.
\end{proof}

We note that $\tilde{\vx}_i$ does not belong to $\V_{h,0}^3(z)$ by \eqref{xi3}. 
However, by integration by parts and using that $\tilde{\vx}_i$ vanishes on 
$\partial \Omega_h(z)$ we do have that $\int_{\Omega_h(z)} \dive \tilde{\vx}_i dx=0$
and hence
\begin{equation} \label{eqn:sumbee}
b_{1i}+b_{2i} + \cdots+b_{Ni}=0 \quad \text{ for } i=1,2. 
\end{equation} 
This also follows by summing \eqref{eqn:beejayeye}.

We can now correct $\tilde{\vx}_i$ to make it belong to  $\V_{h,00}^3(z)$.  We define
\begin{equation*}
\vx_i:=\tilde{\vx}_i- c_{1i} \vc_1-  c_{2i} \vc_2- \cdots - c_{N-1\, i} \vc_{N-1} 
\quad \text{ for } i=1,2,
\end{equation*}
where \Rd
$$
c_{j i}= b_{1i}+b_{2i}+ \cdots+ b_{ji}
      =\frac{1}{3}(y_N-y_{j}) \cdot \cde_i^\perp  ,
$$ \Bk
for $j=1, \cdots , N$. 

In the following result, indices are calculated mod $N$,
and in particular $c_{0,i}=c_{N,i}=0$.
\begin{lemma}
It holds, for $i=1,2$, $\vx_i \in \V_{h,00}^3(z)$ and
\begin{alignat}{1}
\dive(\vx_i)|_{T_j} (z)&=\frac{3}{|T_j|} b_{ji}-12 \cot(\theta_j) 
\Big(\frac{c_{ji}}{|e_j|^2}-\frac{c_{j-1\,i}}{|e_{j-1}|^2}\Big)  
\quad \text{ for } 1 \le j \le N,\label{xiz} \\
\|\nabla \vx_i\|_{L^\infty(\Omega_h(z))} & \le \frac{C}{h_z}. \label{xiz2} 
\end{alignat}
\end{lemma}
The bound \eqref{xiz2} follows from \eqref{xi4} and \eqref{chi5}.

We define for $i=1,2$
\begin{equation}\label{Pvec}
d_{ji}:=(\dive\vx_i)|_{T_j} (z)=\frac{3}{|T_j|} b_{ji}
-12 \cot(\theta_j) \Big(\frac{c_{ji}}{|e_j|^2}-\frac{c_{j-1\,i}}{|e_{j-1}|^2} \Big) 
\quad \text{ for } 1 \le j \le N,
\end{equation}
and 
\begin{equation}\label{Psiz}
d_{j0}:=(\dive\vc)|_{T_j}(z)=\cot(\theta_j) 
\Big(\frac{1}{|e_j|^2}-\frac{1}{|e_{j-1}|^2}
\Big) \quad  \text{ for } 1 \le j \le N.
\end{equation}
Using \eqref{xiz2} and \eqref{chiz2}, we note that for all $1\le j \le N$
\begin{equation}\label{bounddji}
|d_{ji}| \le \frac{C}{h_z^s} \quad \text{ where } s=1 \text{ if } i=1,2 \text{ and } s=2 \text{ if } i=0.
\end{equation}

We can now prove the following important result.
\begin{lemma}\label{even}
Let $z \in \Sh^{1, \text{int}}$ with $\Th(z)=\{ T_1, \ldots, T_N \}$ with $N$ even. Assume that for at least one $i=0,1,2$
\begin{equation}
\decvar_i:=\sum_{j=1}^N (-1)^{j} d_{ji} \neq 0, \label{bigd}
\end{equation}
then $z \in \Lh$. Moreover, in this case,  there exists a constant $C$ depending only 
on the shape regularity constant such that 
\begin{equation}\label{evenbound}
D_z\le C \max_{0 \le i \le 2} \left( 1+ \frac{1}{|\decvar_i|  h_z^s}\right) ,
\end{equation}
where $s=1$ if $i=1,2$ and $s=2$ if $i=0$.
\end{lemma}

\begin{proof}
We set $\vx_0=\vc$. 
Let $i$ be such that \eqref{bigd} holds, 
We let $s_1=0$ and define inductively
\begin{equation*}
s_j= d_{j i}-s_{j-1} \quad \text{ for } 2 \le j \le N.
\end{equation*}
Define 
\begin{equation*}
\vv_1=\frac{-1}{\decvar_i}(\vx_i- s_2 \ww_{2} -\cdots -s_{N-1} \ww_{N-1}-s_{N} \ww_{N}).  
\end{equation*}
We easily have that, using \eqref{Pvec}, \eqref{Psiz}
\begin{equation*}
\dive \vv_1|_{T_j}(z)= \frac{-1}{\decvar_i}(d_{ji}-(s_{j-1}+s_j))=0 \quad \text{ for } 2 \le j \le N,
\end{equation*}
and
\begin{equation*}
\dive \vv_1|_{T_1}(z)=\frac{-1}{\decvar_i}(d_{1i}-s_{N}).
\end{equation*}
We see that $s_N=\sum_{j=2}^N (-1)^j d_{ji}$, we have $d_{1i}-s_{N}= -\sum_{j=1}^N (-1)^j d_{ji}= -\decvar_i$. Hence, 
\begin{equation*}
\dive \vv_1|_{T_1}(z)=1.
\end{equation*}

The following bound follows from \eqref{bounddji}, \eqref{xiz2}, \eqref{chiz2} and \eqref{w7}
\begin{equation}\label{vvbound1}
\|\nabla \vv_1\|_{L^\infty(\Omega_h(z))} \le \frac{C}{|\decvar_i|  h_z^s}.
\end{equation}

Then, we define inductively 
\begin{equation*}
\vv_\ell= \ww_\ell-\vv_{\ell-1} \quad \text{ for } 2 \le \ell \le N.
\end{equation*}
We then easily see that
\begin{equation}
\dive \vv_\ell|_{T_j}(z)= \delta_{\ell j} \quad \text{ for } 1 \le \ell,j \le N. 
\end{equation}
Moreover, from \eqref{w7} and \eqref{vvbound1} we have 
\begin{equation*}\label{vvboundi}
\|\nabla \vv_\ell\|_{L^\infty(\Omega_h(z))} \le C\left( 1+ \frac{1}{|\decvar_i|  h_z^s}\right).
\end{equation*}

For an arbitrary $a=(a_1, a_2, \ldots, a_N) \in W(z)$ we simply define
\begin{equation*}
\vv=\sum_{\ell=1}^N a_\ell \vv_\ell. 
\end{equation*}
Then, $\vv$ satisfies \eqref{aux134} and hence $z \in  \Lh$. By \eqref{vvboundi} we have 
\begin{equation*}
\|\nabla \vv\|_{L^\infty(\Omega_h(z))} \le  C\left( 1+ \frac{1}{|\decvar_i|  h_z^s}\right) \max_\ell |a_\ell|,
\end{equation*}
which proves \eqref{evenbound}. 
\end{proof}

\subsection{Simplification of condition \eqref{bigd}}

In the case $i=0$, we have
\begin{equation}\label{eqn:simplebigd}
\begin{split}
\decvar_0&=\sum_{i=1}^N (-1)^j d_{j0}=
\sum_{i=1}^N (-1)^j \cot(\theta_j)\big(|e_j|^{-2}-|e_{j-1}|^{-2}\big) \\
&=\sum_{i=1}^N (-1)^j \big(\cot(\theta_j)+\cot(\theta_{j+1})\big)|e_j|^{-2}.
\end{split}
\end{equation}
Thus we see that generically this is nonzero, since the lengths $|e_j|$
can be chosen independently of the angles $\theta_j$.
For $i=1,2$, we have
\begin{equation}\label{eqn:beetermone}
\begin{split}
\sum_{j=1}^N (-1)^j \frac{3}{|T_j|} b_{ji} &=
\sum_{j=1}^N (-1)^j \frac{1}{|T_j|} \big(y_{j-1}-y_j\big)\cdot\cde_i^\perp \\
&=-\sum_{j=1}^N (-1)^j \Big(\frac{1}{|T_j|}+\frac{1}{|T_{j+1}|}\Big) y_j\cdot\cde_i^\perp \\
\end{split}
\end{equation}

For $i=1,2$, we have
\begin{equation}\label{eqn:pretypeone}
\begin{split}
-12 \sum_{j=1}^N (-1)^j 
\cot(\theta_j) \Big(\frac{c_{j,i}}{|e_j|^2}&-\frac{c_{j-1,i}}{|e_{j-1}|^2} \Big) 
=-12\sum_{j=1}^N (-1)^j \cot(\theta_j)
\Big(\frac{c_{j,i}}{|e_j|^2}-\frac{c_{j-1,i}}{|e_{j-1}|^2} \Big) \\
&=4\sum_{j=1}^N (-1)^j \cot(\theta_j)
\Big(\frac{y_{j}-y_N}{|e_j|^2}-\frac{y_{j-1}-y_N}{|e_{j-1}|^2} \Big)\cdot\cde_i^\perp \\
&=4\sum_{j=1}^N (-1)^j \cot(\theta_j)
\Big(\frac{y_{j}}{|e_j|^2}-\frac{y_{j-1}}{|e_{j-1}|^2} \Big)\cdot\cde_i^\perp \\
 &\quad -4\bigg(\sum_{j=1}^N (-1)^j \cot(\theta_j)
\Big(\frac{1}{|e_j|^2}-\frac{1}{|e_{j-1}|^2} \Big)\bigg)y_N\cdot\cde_i^\perp \\
&=4\sum_{j=1}^N (-1)^j \cot(\theta_j)
\Big(\frac{y_{j}}{|e_j|^2}-\frac{y_{j-1}}{|e_{j-1}|^2} \Big)\cdot\cde_i^\perp \\
&\quad -4\bigg(\sum_{j=1}^N (-1)^j d_{j0} \bigg) y_N\cdot\cde_i^\perp \\
&=4\sum_{j=1}^N (-1)^j
    \frac{\cot(\theta_j)+\cot(\theta_{j+1})}{|e_j|^2}y_{j}\cdot\cde_i^\perp
-4\decvar_0 \, y_N\cdot\cde_i^\perp .
\end{split}
\end{equation}
Therefore, for $i=1,2$,
\begin{equation}\label{eqn:alldeemone}
\begin{split}
\decvar_i&=
\sum_{j=1}^N (-1)^j\bigg(4 \frac{\cot(\theta_j)+\cot(\theta_{j+1})}{|e_j|^2}
-\Big(\frac{1}{|T_j|}+\frac{1}{|T_{j+1}|}\Big) \bigg) y_{j}\cdot\cde_i^\perp
-4\decvar_0 \, y_N\cdot\cde_i^\perp .
\end{split}
\end{equation}

\section{Meshes where \eqref{bigd} fails to hold}

Lemma \ref{even} gives sufficient conditions for an interior vertex $z$ with 
an even number of triangles to be a local interpolating vertex (i.e., $z \in \Lh$). 
We see that \eqref{bigd} is a mild constraint and that a generic vertex will 
satisfy \eqref{bigd}, however, there are important examples of vertices which do not 
satisfy \eqref{bigd} and perhaps are not local interpolating vertices.
Here we present some examples.

\subsection{Regular $N$-gon with $N$ even}

Suppose that $\Omega_h(z)$ is a triangulated regular $N$-gon with $N$ even.
More precisely, we assume that $\Omega_h(z)$ is subdivided by $N$ similar
triangles, with edge lengths $|e_1|=|e_2|= \cdots=|e_N|$ and interior angles
$\theta_1=\theta_2= \cdots=\theta_N$.
Then we can show that \eqref{bigd} does not hold. 

First of all, the condition on the edge lengths alone implies that $d_{j0}=0$
for all $j=1,\dots,N$. Thus $\decvar_0=0$.

Now consider $i=1,2$.
The vertices of the regular $N$-gon can be written as
$$
y_k=h e^{\iota 2\pi k/N},
$$
for some $h$. Here we make the the standard association $e^{\iota \theta}$ with the vector $[\cos \theta, \sin \theta]^t$.
We conclude that, for $N$ even, 
\begin{equation}\label{eqn:symcondc}
h^{-1}\sum_{j=1}^N (-1)^j y_{j} = \sum_{j=1}^{N/2} e^{\iota 4\pi j/N} 
    - \sum_{j=1}^{N/2} e^{\iota 2\pi (2j-1)/N} 
= \big(1-e^{-\iota 2\pi/N}\big)\sum_{j=1}^{N/2} e^{\iota 4\pi k/N} =\bfz.
\end{equation}

Since all the triangles are the same, using \eqref{eqn:alldeemone} and that we showed $\decvar_0=0$ we have for $i=1,2$
\begin{equation*}
\decvar_i= \left( \frac{8\cot(\theta_1)}{{|e_1|^2}}-\frac{2}{|T_1|}\right) \sum_{j=1}^N (-1)^j y_{j}\cdot\cde_i^\perp=0,
\end{equation*}
where we used \eqref{eqn:symcondc}. Thus \eqref{bigd} fails for $i=1,2$ as well, and Lemma \ref{even} cannot
be used.

\subsection{Three lines mesh}

For the three-lines mesh generating a regular hexagonal pattern as shown on the 
left in Figure \ref{fig:hexanddiagonal}, the condition \eqref{bigd} also fails 
for $i=0,1,2$ since each vertex is at the center of a regular hexagon.

\section{Crossed triangles}

Consider the mesh shown on the right in Figure \ref{fig:hexanddiagonal}.
Half of the vertices are at the center of a regular 4-gon, but these are
singular vertices, so these are all local interpolating vertices.
For the other vertices, at the center of a non-regular 8-gon, we can argue
as follows.
Since the interior angles are all the same ($\pi/4$), and $\cot(\pi/4)=1$,  we have using \eqref{eqn:simplebigd}

\begin{equation}\label{eqn:octcondc}
\decvar_0=2\sum_{j=1}^8 (-1)^j\frac{1}{|e_j|^2}.
\end{equation}
Let $L$ be the length of the smallest edge: $L=\min\set{|e_j|}{j=1,\dots,8}$.
Then the longest edge length is $\sqrt{2} L$, and the edge lengths $|e_j|$ 
alternate $L,\sqrt{2} L,L,\sqrt{2} L,\dots$.
Thus
$$
\sum_{j=1}^8 (-1)^j\frac{1}{|e_j|^2}=\pm\frac{2}{L^2},
$$
depending on where we start the counting.
Therefore condition \eqref{bigd} holds at these vertices, and thus Lemma \ref{even} 
can be applied to conclude that these vertices are also local interpolating vertices.
Thus all of the vertices in the mesh shown on the right in 
Figure \ref{fig:hexanddiagonal} are in $\Lh$.

\begin{figure}
\centerline{ \includegraphics[scale=1.1]{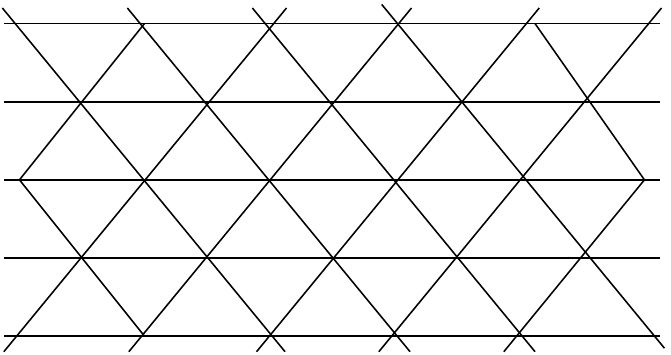}
\qquad   \includegraphics[scale=.6]{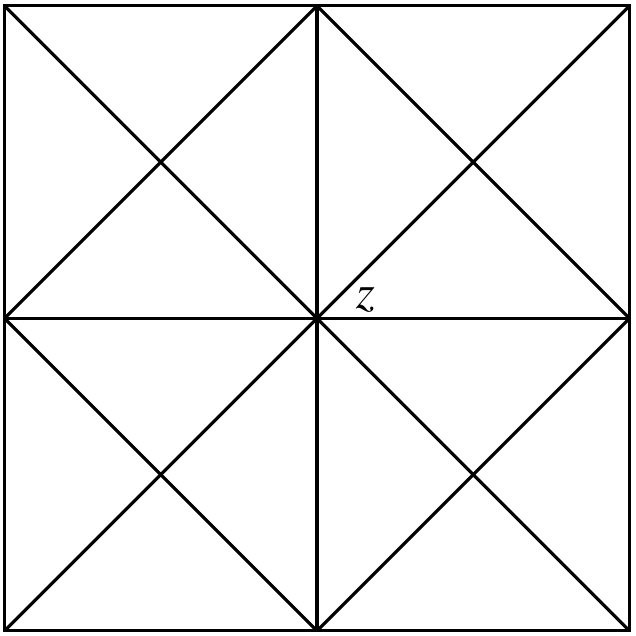}}
\caption{For the mesh on left no interior vertex satisfies \eqref{bigd}. All interior vertices on the mesh on the right belong to $\Lh$.} 
\label{fig:hexanddiagonal}
\end{figure}

\section{Inf-sup stability when all interior vertices belong to  $\Lh$ }
\label{sec:kaythree}

In the previous section we have identified many local interpolating vertices. In particular, singular vertices and interior vertices with odd number of triangles are local interpolating vertices; see Lemmas \ref{lemma0} and \ref{odd}. If $z$ is an interior non-singular vertex with even number of triangles then Lemma \ref{even} gives sufficient conditions for it to be a local interpolating vertex.  None of the above examples address boundary vertices that are non-singular.  In this section we will show how to interpolate at those vertices but in a non-local way. Then, using that result and assuming that $\Sh^{\text{int}} \subset \Lh$ we will prove inf-sup stability. In the next section we will address $\Sh^{\text{int}} \not \subset \Lh$.

For vertices that are not local interpolating vertices we can still interpolate there 
but with a side effect of polluting a neighboring vertex. 
In other words, the vector field will not belong to $\V_{h,00}^3(z)$. 
To do this, we will need to define a piecewise cubic function that has 
average zero on edges. 
For every $z \in \Sh^\partial$ and interior edge $e \in \Eh$ with $e=\{z, y\}$ we set
\begin{equation}\label{defkappa}
\kappa_e^{z} = \eta_e^z- \frac{1}{2} \psi_z\psi_y
             = \psi_z^2\psi_y - \frac{1}{2} \psi_z\psi_y.
\end{equation}
The function $\kappa_e^{z}$ will play the same role as $\gamma_e^z$ 
in \cite{guzmanscottdegfour} but the difference is that the added term 
in $\kappa_e^z$ is piecewise quadratic. 
Let $T_1, T_2 \in \Th(z)$ be two triangles that have $e$ as an edge,
and let $\theta_i$ be the angle between the edges of $T_i$ emanating from $z$. 

Then we can easily verify the following:
\begin{subequations}\label{kappa}
\begin{alignat}{1}
& \text{ supp }   \kappa_e^{z} \subset T_1 \cup T_2,  \label{kappa1} \\
& (\nabla\kappa_e^{z}) (\sigma )=0 \quad \text{ for  } \sigma \in \Sh,    
\; \sigma \neq z \text{ and } \sigma \neq  y, \label{kappa2} \\
 &\int_e \kappa_e^{z} \, ds =0 \label{kappa3} \\
 &\nabla \kappa_e^{z}|_{T_i}(z) = \frac{1}{2}\nabla\psi_y|_{T_i} , \quad \nabla \kappa_e^{z}|_{T_i}(y) = -\frac{1}{2}\nabla\psi_z|_{T_i} \quad \text{ for } i=1,2  \label{kappa4}.
\end{alignat}
\end{subequations}

Using these functions we can prove the following result.

\begin{lemma}\label{boundary}
For every $p \in Q_h^{2}$ and $z \in \Sh^\partial$ there exists a 
$\vv \in \V_{h,0}^3(z)$ such the following properties hold:
\begin{subequations}\label{lv}
\begin{alignat}{1}
&\dive \vv (\sigma)= 0 \quad \text{ for all } \sigma \in \Sh^{\partial},  
\; \sigma \neq z,
\label{lv1}  \\
& \dive \vv|_T (z)= p|_T(z) \quad \text{ for all } T \in \Th(z)
\label{lv2}.
\end{alignat}
\end{subequations}
If $z$ is non-singular 
\begin{equation}
\|\nabla \vv \|_{L^2(\Omega_h(z))} 
\le C \Big(\frac{1}{\Theta(z)}+1\Big)\|p\|_{L^2(\Omega_h(z))}.
\end{equation}
If $z$ is singular  
\begin{equation}
 \|\nabla \vv \|_{L^2(\Omega_h(z))} \le C  \|p\|_{L^2(\Omega_h(z))}.
\end{equation}
The constant $C$ only depends only on the shape regularity.
\end{lemma}

\begin{proof}
If $z$ is singular, the result follows from Lemma \ref{lemma0}, so now assume
that $z$ is non-singular. Enumerate the triangles such that $T_1$ and $T_N$ each have a boundary edge,
and $T_j, T_{j+1}$ share an edge  $e_j=\{z, y_j\}$,  for $j=1, \ldots N-1$. 
Let $\theta_j $ denote the angle between the edges of $T_j$ originating from $z$. 
Also let $\n_j$ be the normal to $e_j$ out of $T_j$ and $\tv_j$ be tangent 
to $e_j$ pointing away from $z$.
Let $1 \le s \le N-1$ be such that $ |\sin(\theta_s+ \theta_{s+1})|=\Theta(z)$. We will define vector fields $\vv_1, \ldots, \vv_N$. We start by defining $\vv_s$.
\begin{equation*}
\vv_{s}= \frac{ 2 |e_s| \, \sin(\theta_s)}{\sin(\theta_s+ \theta_{s+1})}   \tv_{s+1} \kappa_{e_s}^z .
\end{equation*}
Then, we can easily show that
\begin{equation*}
\dive \vv_{s}|_{T_i}(z)= \delta_{i,s} \text{ for } 1 \le i \le N,
\end{equation*}
where $\delta_{i,s}$ is the Kronecker $\delta$.
Indeed, 
\begin{equation*}
\dive \vv_{s}|_{T_i}(z)
=\frac{ 2 |e_s| \, \sin(\theta_s)}{\sin(\theta_s+ \theta_{s+1})} \tv_{s+1}\cdot \nabla \kappa_{e_s}^z|_{T_i}(z) 
=\frac{   |e_s| \, \sin(\theta_s)}{\sin(\theta_s+ \theta_{s+1})} \tv_{s+1} \cdot \nabla \psi_{y_s}|_{T_i}(z) .\\
\end{equation*}
The result follows after using \eqref{tpsi} and \eqref{hTj} to calculate  $\psi_{y_s}|_{T_i}$ and using basic trigonometry. Then we can define inductively for $s+1 \le j \le N$. 
\begin{equation*}
\vv_{j}=\ww_{e_j}^z-\vv_{j-1}.
\end{equation*}
Also, for $1 \le j \le s-1$ we define 
\begin{equation*}
\vv_{j}=\ww_{e_j}^z-\vv_{j+1}.
\end{equation*}
Hence, we have the following property:
\begin{equation}
\dive \vv_{j}|_{T_i}(s)= \delta_{i,j} \text{ for } 1 \le i,j \le N.
\end{equation}
We then define 
\begin{equation*}
\vv= \sum_{j=1}^N p|_{T_j}(z) \vv_{j}.
\end{equation*}
The stated conditions on $\vv$ are easily verified.
\end{proof}

We can then use the above lemma to prove the following result. We define 
 \begin{equation*}
 \Theta_{\min, \partial}= \min_{z \in \Sh^{1, \partial}} \Theta(z).
 \end{equation*}

\begin{lemma}\label{blemma}
For every $p \in Q_h^{2}$  there exists a  $\vv \in \V_{h,0}^3$ such that 
\begin{alignat}{1}
& (\dive \vv-p)(z)=0 \quad \text{ for all } z \in \Sh^\partial,  
\end{alignat}
and
\begin{equation*}
 \|\nabla \vv \|_{L^2(\Omega)} \le C_b \Big(\frac{1}{\Theta_{\min,\partial}}+1\Big) \|p\|_{L^2(\Omega)}.
\end{equation*}
The constant $C_b$ only depends on the shape regularity.
\end{lemma}

\begin{proof}
Let $p \in Q_h^2$. For every  $z \in \Sh^{\partial}$  let $\vv^z$ be the vector field from Lemma \ref{boundary} then we set 
\begin{equation*}
\vv=\sum_{z \in \Sh^\partial} \vv_z.
\end{equation*}
Then, it is easy to verify the conditoins on $\vv$. 

\end{proof}

We can now prove the main result of the section.

\begin{theorem}\label{thm1}
Assume that $\Sh^{\text{int}} \subset \Lh$. Then, for every $p \in Q_h^{2}$ there exists a $\vv \in  V_{h,0}^3$ satisfying
\begin{alignat}{1}
(\dive \vv-p)(\sigma)&=0 \quad \text{ for all } \sigma \in \Sh,
\end{alignat}
with the bound
\begin{equation*}
\|\vv\|_{H^1(\Omega)}  \le  C (1+D) (1+\frac{1}{\Theta_{\min, \partial}}) \|p\|_{L^2(\Omega)}.
\end{equation*}
where  $D=\max_{z \in \Lh} D_z$.
\end{theorem}

\begin{proof}
Let $\vv_1$ be from Lemma  \ref{blemma} and let $p_1=p-\dive \vv_1$ and we note that $p_1 \in Q_h^2$ with $p_1$ vanishing on boundary vertices. Since $\Sh^{\text{int}} \subset \Lh$ for every $z \in \Sh^{\text{int}}$ there exists a $\vv^z \in \V_{h,00}^3(z)$ such that
\begin{equation*} 
(\dive \vv^z-p_1)(z)=0 \  
\end{equation*}
with
\begin{equation*}
\|\nabla \vv^z\|_{L^\infty(\Omega_h(z))} \le D_z \max_{T \in \Th(z)} |p_1|_T(z)| .
\end{equation*}
Using an inverse estimate we can show 
\begin{equation*}
\|\nabla \vv^z\|_{L^2(\Omega_h(z))} \le C D_z \|p_1\|_{L^2(\Omega_h(z))}.
\end{equation*}
If we set 
\begin{equation*}
\vv_2= \sum_{z \in \Sh^{\text{int}}} \vv^z
\end{equation*}
and set $\vv=\vv_1 +\vv_2$, then the desired conditions on $\vv$ are met.
\end{proof}

Using Lemma \ref{lemma4} we have the following corollary.

\begin{corollary}
Assume that $\Sh^{\text{int}} \subset \Lh$, then the inf-sup condition holds for $Q_h^2 \times \V_h^3$ with constants given by Lemma \ref{lemma4} and Theorem \ref{thm1}. 
\end{corollary}

\section{Inf-sup stability: the general case} 
\label{sec:gencase}

As mentioned above for most meshes $\Sh^{\text{int}} \subset \Lh$ and hence by the previous section one can prove inf-sup stability. However, for very important meshes, such as the diagonal mesh, none of the interior vertices belong to $\Lh$ (i.e. $\Lh \cap \Sh^{\text{int}} = \emptyset$). However, if some interior nodes belong to $\Lh$ then it might hold that $\dive \V_h^3=Q_h^2$ and we can give a bound for inf-sup constant. To do this we use a concept of a tree and paths. We consider the mesh as a graph and consider trees and paths that are subgraphs of the mesh. Precise statements are given below.

\subsection{Paths and trees in a mesh}

We will prove that if there is a tree of the mesh $\Th$ with root in  $\Lh$  satisfying certain mild conditions then we can interpolate a pressure on all the vertices of the tree. We start with some preliminary results. Let $T_1, T_2 \in \Th(z)$ be two triangles that have $e$ as an edge and let $\phi_i$ be the angle between the edges of $T_i$ emanating from $z$.  Then we define
\begin{equation}\label{Mez}
M_e^z= \cot(\phi_1)+\cot(\phi_2).
\end{equation}
Note that $M_e^z=0$ if $\phi_1+\phi_2= \pi$.

\begin{figure}
\centerline{\includegraphics[width=2.7in]{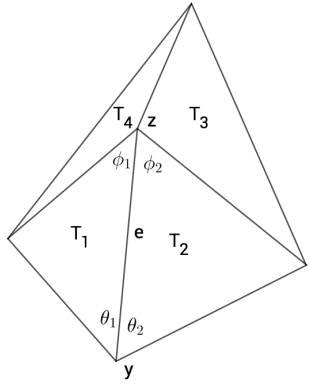}}
\caption{Illustration for Lemma \ref{lemmazy}, with $N=4$.}
\label{lemmazyill}
\end{figure}

\begin{lemma}\label{lemmazy}
Let $z,y \in \Sh$ and $e=\{z,y\}\in\Eh$ and suppose that $M_e^z\neq 0$.
Let $\Th(z)=\{ T_1, T_2, \ldots, T_N\}$ and suppose $T_1$ and $T_2$ are
the two triangles that share $e$ as an edge and let $\theta_i$ be the angles of $T_i$ originating from $y$ for $i=1,2$; see Figure \ref{lemmazyill}. Let $a=(a_1, a_2, \ldots, a_N) \in \R^N$ and define the alternating sum $s(a)= \sum_{j=1}^N (-1)^{j-1} a_j$. Then, there exists a $\vv \in \V_{h,0}^3(z)$
\begin{subequations}\label{zy}
\begin{alignat}{1}
\dive \vv|_{T_i} (z)&= a_i \quad \text{ for all } 1\le i \le N \label{zy1}\\
\dive \vv(\sigma)&=0 \quad \text{ for all } \sigma \in \Sh,\, \sigma \neq z,\, \sigma \neq y, \label{zy2} \\
\dive \vv|_{T_1}(y)&= -s(a) \frac{\cot(\theta_1)}{M_e^z}, \quad  \dive \vv|_{T_2}(y)=  s(a) \frac{\cot(\theta_2)}{M_e^z} \label{zy3} \\
\|\nabla \vv\|_{L^\infty(\Omega_h(z))} \le &   \frac{C }{|M_{e}^z|}
      \max_{1 \le i \le N}  |a_i|. \label{zy5}
\end{alignat}
\end{subequations}
The constant $C$ depends only on the shape regularity. 
\end{lemma}

\begin{proof}
We prove the result in the case $z$ is an interior vertex. The case $z$ is a boundary vertex is similar. We use the notation from Section \ref{scheme}.  We need to define some auxiliary vector fields.  First let 
\begin{equation*}
\rr= 2 |e_1| \kappa_{e_1}^z \n_1,
\end{equation*}
where we recall that the definition of $\kappa_{e_1}^z$ is given in \eqref{defkappa}. Using, \eqref{kappa1} and \eqref{kappa3} it is easy to show that $\rr\in\V_{h,0}^3(z)$.
Then, using \eqref{kappa2} we have  
\begin{equation*}
\dive \rr|_{T_i}(\sigma)=0 \quad \text{ for all } \sigma \in \Sh, \, \sigma \neq z,
\, \sigma \neq y. \\
\end{equation*}
Also, using \eqref{kappa4}, \eqref{nablapsi}, \eqref{hTj} we can show that 
\begin{equation*}
\dive \rr|_{T_1}(z)= \cot(\phi_1), \quad \dive \rr|_{T_2}(z)= -\cot(\phi_2),
\end{equation*}
and
\begin{equation*}
\dive \rr|_{T_1}(y)= -\cot(\theta_1), \quad \dive \rr|_{T_2}(z)= \cot(\theta_2).
\end{equation*}
If we let 
\begin{equation*}
\vv_1= \frac{1}{M_e^z}\big(\rr+\cot(\phi_2) \ww_{1}\big),
\end{equation*}
then using the properties of $\ww_1$ (e.g.~\eqref{w}) we have the following 

\begin{equation*}
\dive \vv_1|_{T_1}(z)=1, \quad \dive \vv_1|_{T_2}(z)=0,
\end{equation*}
\begin{equation*}
\dive \vv_1|_{T_1}(y)=-\frac{\cot(\theta_1)}{M_e^z}, \quad \dive \vv_1|_{T_2}(y)=\frac{\cot(\theta_2)}{M_e^z},
\end{equation*}
and
\begin{equation*}
\|\nabla \vv_1\|_{L^\infty(\Omega_h(z))} \le \frac{C}{|M_e^z|},  
\end{equation*}
where the constant $C$ only depends on the shape regularity constant. 

Next, we define inductively 
\begin{equation*}
\vv_j=\ww_{j-1}-\vv_{j-1} \text{ for }  2 \le j \le N.
\end{equation*}
Using the properties of $\vv_1$ just proved together with \eqref{w},
we can show the following for all $1 \le j \le N$ :
\begin{subequations}
\begin{alignat}{1}
\vv_j \in & \V_{h,0}^3(z), \\
\dive \vv_j (\sigma) &= 0  \quad \text{ for all } \sigma \in \Sh, \, \sigma \neq z, \, \sigma \neq y, \\
\dive \vv_j|_{T_i} (z)&= \delta_{ij}  \quad \text{ for all } 1 \le i, j \le N, \\
\|\nabla \vv_j\|_{L^\infty(\Omega_h(z))} \le & C\Big(1+\frac{1}{|M_e^z|}\Big) \le \frac{C}{|M_e^z|},
\end{alignat}
\end{subequations}
where we used that $|M_e^z| \le c$ is bounded where the  constant $c$ depends on the shape regularity.
By the definition of $\vv_j$ and using that $\dive \ww_j(y)=0$, we note that $\dive \vv_j (y)=(-1)^{j-1} \dive \vv_1(y)$.

We now set
\begin{equation*}
\vv=\sum_{j=1}^N a_j \vv_j.
\end{equation*}
We easily see  conditions \eqref{zy} hold. 
\end{proof}

We can apply the previous result repeatedly to generalize the result for a path. 
\begin{Def}\label{defpath}
Given $z, y \in \Sh$   $P=\{y_0, y_1, \ldots, y_{L}\} $ is a path between $y_0=z$, $y_L=y$ if $e_i=\{ y_{i-1}, y_{i}\} \in \Eh$  and $y_i \neq y_j$ for $i \neq j$.  We say that the path is {\it acceptable } if $M_{e_i}^{y_{i-1}} \neq 0$ for $1 \le i \le L$.
\end{Def}
See Figure \ref{path} for an illustration. For an acceptable path $P$ as in Definition \ref{defpath} we define  for $1 \le j \le L-1 $ 
 \begin{equation*}
\tilde{\rho}_{z,y_j} := \frac{M_{e_1}^{y_1}}{M_{e_1}^{y_0}} \frac{M_{e_2}^{y_2}}{M_{e_2}^{y_1}}  \cdots \frac{M_{e_{j}}^{y_{j}}}{M_{e_{j}}^{y_{j-1}}}. 
\end{equation*}
We also let $\tilde{\rho}_{z,z}=1$. Moreover, we define
\begin{equation}\label{rho}
\rho_{z,y_{j+1}}:=\frac{\tilde{\rho}_{z,y_j}}{M_{e_{j+1}}^{y_{j}}} .
\end{equation}

Finally, we let 
\begin{equation}\label{rhoP}
\rho(P)= \max_{1 \le j \le L}|\rho_{z,y_j}|.
\end{equation}
Also for any collection of vertices $P$ we define $\Th(P)=\{ T \in \Th: T \in \Th(y) \text{  for some  } y \in P \}$. We define $\Omega_h(P)= \bigcup_{T \in \Th(P)}  T$. 

\begin{figure}
\centerline{\includegraphics[width=6in]{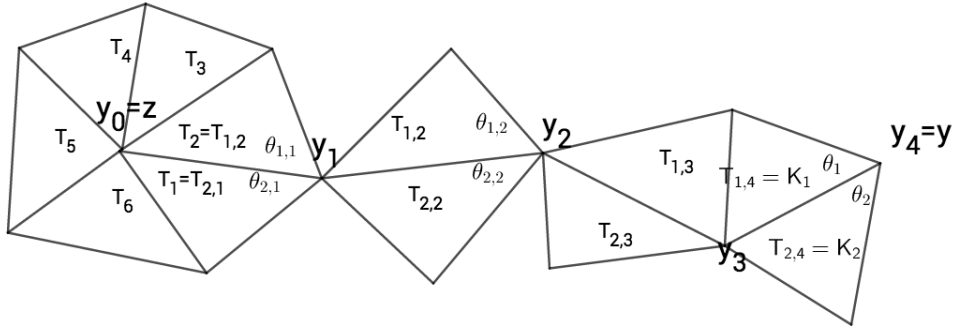}}
\vspace{-140pt}
\caption{Illustration of acceptable path, with $N=6$, L=4.}
\label{path}
\end{figure}

\begin{lemma}\label{lemmapath}
Suppose that $z,y \in \Sh$ and let $\Th(z)=\{ T_1, T_2, \ldots, T_N\}$. Let  $P=\{y_0, y_1, \ldots, y_{L}\} $ with $y_0=z$, $y_L=y$ be an {\it acceptable} path.  Also, we denote by $K_1$ and $K_2$ the two triangles that share $e_L$ and $\theta_i$ the corresponding angles. For every $a=(a_1, \ldots, a_N) \in \R^N$ there exists a $\vv \in \V_{h,0}^3$ with support in $\Omega_h(P)$ such that
\begin{subequations}\label{path}
\begin{alignat}{1}
\dive \vv|_{T_i} (z)&= a_i \quad \text{ for all } 1\le i \le N, \label{path1}\\
\dive \vv(\sigma)&=0 \quad \text{ for all } \sigma \in \Sh, \sigma \neq z, \sigma \neq y, \label{path2} \\
\dive \vv|_{K_1}(y)&= \pm s(a)\tilde{\rho}_{z,y_{L-1}} \frac{\cot(\theta_1)}{M_{e_{L}}^{y_{L-1}}}, \quad  \dive \vv|_{K_2}(y)=  \mp s(a) \tilde{\rho}_{z,y_{L-1}}  \frac{\cot(\theta_2)}{M_{e_{L}}^{y_{L-1}}} \label{path3} \\
\dive \vv|_{T}(y)&= 0 \quad \text{ for all } T \in \Th(y),\, T \neq K_1,\, T \neq K_2, \label{path4} \\
\|\nabla \vv\|_{L^\infty(\Omega_h(P))} \le &  C \rho(P)  \max_{1 \le i \le N}  |a_i|. \label{path5}  
\end{alignat}
\end{subequations}
Here $s(a)=\sum_{j=1}^N (-1)^j a_j$. The constant $C$ only depends on the shape regularity of the mesh.  
\end{lemma}

\begin{proof}
Let $a=(a_1, \ldots, a_N) \in \R^N$ be given.  We will use the following notation: $e_j$ is the edge with vertices $y_{j-1}, y_j$. We assume that $\Th(y_j)=\{ T_{1,j}, T_{2,j}, \ldots, T_{N_j, j}\} $ and such that $T_{1,j}, T_{2,j} $ share $e_j$ as  a common edge. The corresponding angles are denoted by $\theta_{1,j}, \theta_{2,j}$.  Note that $K_1=T_{1, L}$ and $K_2=T_{2,L}$. By Lemma \ref{lemmazy} we have  $\rr_1 \in \V_{h,0}^3(y_0)$ 
\begin{subequations}\label{rr}
\begin{alignat}{1}
\dive \rr_1|_{T_i} (y_0)&= a_i \quad \text{ for all } 1\le i \le N \label{rr1}\\
\dive \rr_1(\sigma)&=0 \quad \text{ for all } \sigma \in \Sh,\, \sigma \neq y_0,\, \sigma \neq y_1, \label{rr2} \\
\dive \rr_1|_{T_{1,1}}(y_1)&= -s(a) \frac{\cot(\theta_{1,1})}{M_{e_1}^{y_0}}, \quad  \dive \rr_1|_{T_{2,j}}(y_1)=  s(a) \frac{\cot(\theta_{2,1})}{M_{e_1}^{y_0}} \label{rr3} \\
\dive \rr_1|_{T_{i,1}}(y_1)&= 0 \quad \text{ for all } 3 \le i \le N \label{rr4} \\
\|\nabla \rr_1\|_{L^\infty(\Omega_h(y_0))} & \le  \frac{C }{|M_{e_1}^{y_0}|}
      \max_{1 \le i \le N}  |a_i|. \label{rr5}
\end{alignat}
\end{subequations}
Now suppose that we have constructed $\rr_j \in \V_{h,0}^3(y_{j-1})$ for $j=2, \ldots, \ell$ with the following properties
\begin{subequations}\label{rrr}
\begin{alignat}{1}
  \left(\dive \rr_j+ \dive \rr_{j-1}\right)(y_{j-1})&=0,   \label{rrr1}\\
\dive \rr_j(\sigma)&=0 \quad \text{ for all } \sigma \in \Sh,\, \sigma \neq y_{j-1},\, \sigma \neq y_{j}, \label{rrr2} \\
\dive \rr_j|_{T_{1,j}}(y_j)= \pm s(a)\tilde{\rho}_{z,y_{j-1}}\frac{\cot(\theta_{1,j})}{M_{e_j}^{y_{j-1}}}, & \quad  \dive \rr_j|_{T_{2,j}}(y_j)= \mp s(a)\tilde{\rho}_{z,y_{j-1}} \frac{\cot(\theta_{2,j})}{M_{e_j}^{y_{j-1}}} \label{rrr3} \\
\dive \rr_j|_{T_{i,j}}(y_j)&= 0 \quad \text{ for all } 3 \le i \le N_j, \label{rrr4} \\
\|\nabla \rr_j\|_{L^\infty(\Omega_h(y_{j-1}))} & \le C |\rho_{z,y_j}| \max_{1 \le i \le N}  |a_i|. \label{rrr5}
\end{alignat}
\end{subequations}

Setting  $\tilde{a}_i=\dive \rr_{\ell}|_{T_{i,\ell}}(y_\ell)$ for $1 \le i \le N_\ell$  and  using \eqref{rrr3} and \eqref{rrr4} we have
\begin{alignat}{1}
s(\tilde{a})&=\sum_{i=1}^{N_{\ell}} (-1)^{i-1} \tilde{a}_i= \tilde{a}_1- \tilde{a}_{2}= \pm s(a)\tilde{\rho}_{z,y_{\ell-1}} \frac{\cot(\theta_{1,\ell})}{M_{e_{j}}^{y_{\ell-1}}} \pm s(a)\tilde{\rho}_{z,y_{\ell-1}} \frac{\cot(\theta_{2,\ell})}{M_{e_{\ell}}^{y_{\ell-1}}},  \nonumber\\
&=\pm s(a) \tilde{\rho}_{z,y_{\ell-1}} \frac{\cot(\theta_{1,\ell})+\cot(\theta_{2,\ell})}{M_{e_{j}}^{y_{\ell-1}}} =\pm s(a) \tilde{\rho}_{z,y_{\ell-1}} \frac{M_{e_\ell}^{y_\ell}}{M_{e_{\ell}}^{y_{\ell-1}}}=\pm s(a) \tilde{\rho}_{z,y_\ell}. \label{771}
\end{alignat}

Hence, using Lemma \ref{lemmazy} we can find $\rr_{\ell+1}  \in \V_{h,0}^3(y_\ell)$ such that 
\begin{subequations}\label{rell}
\begin{alignat}{1}
\quad \dive \rr_{\ell+1}|_{T_{i,\ell}} (y_\ell)&= -\tilde{a}_i \quad \text{ for all } 1 \le i \le N_\ell \\
\dive \rr_{\ell+1}(\sigma)&=0 \quad \text{ for all } \sigma \in \Sh, \sigma \neq y_\ell, \sigma \neq y_{\ell+1},  \\
\dive \rr_{\ell+1}|_{T_{1,\ell+1}}(y_{\ell+1})&= \pm s(\tilde{a}) \frac{\cot(\theta_{1,\ell+1})}{M_{e_{\ell+1}}^{y_\ell}}, \quad  \dive \rr_{\ell+1}|_ {T_{2,\ell+1}}(y_{\ell+1})=  \mp s(\tilde{a}) \frac{\cot(\theta_{2,\ell+1})}{M_{e_{\ell+1}}^{y_\ell}}  \\
\|\nabla \rr_{\ell+1} \|_{L^\infty(\Omega_h(y_\ell))} \le &  \frac{C}{|M_{e_{\ell+1}}^{y_\ell}|}
    \max_{1 \le i \le N_\ell}  |\tilde{a}_i|  \le \frac{C}{|M_{e_{\ell+1}}^{y_\ell}|} |s(a)| |\tilde{\rho}_{z,y_\ell}|=C |s(a)||\rho_{z,y_{\ell+1}}|.   
\end{alignat}
\end{subequations}
Hence, using \eqref{rell} and \eqref{771} we get that \eqref{rrr} holds for $j=\ell+1$ and hence by induction \eqref{rrr} holds for $2 \le j \le L$.

We let $\vv=\rr_1+ \rr_2+ \cdots+ \rr_L$. We then easily see that \eqref{path1}-\eqref{path4} hold from \eqref{rrr} and \eqref{rr}. To prove \eqref{path5} we use that 
\begin{equation*}
\|\nabla \vv\|_{L^\infty(\Omega_h(P))} \le  \max_{1 \le i \le L} \|\nabla \vv\|_{L^\infty(\Omega_h(y_i))}. 
\end{equation*}
Since finitely many $\{\rr_j\}$ have support in $\Omega_h(y_i)$ the result follows from  \eqref{rrr5} and \eqref{rr5}.
\end{proof}

The next task is to interpolate values on a tree of the mesh. We need some notation.  
\begin{definition}
We say that $\tr(r):=  Y \times E$ with $Y= \bigcup_{ 0 \le j \le M} Y_j$ and $E=\bigcup_{ 1 \le j \le M}  E_j $ is a tree of $\Th$ with root $r$ if the following hold
\begin{alignat*}{1}
 1)&\  Y_0=\{r\}, \\
 2)&\  \text{for every } 1 \le j \le M,\, Y_j \subset \Sh, E_j \subset \Eh  \text{ and } |Y_j|=| E_j |, \\
 3)&\  \text{for every } y \in Y_j \text{ there is a  unique } e \in E_j \text{ such that } e=\{ y, z\} \text{ for some } z \in Y_{j-1}, \\
 4)&\ Y_j \cap Y_i = \emptyset  \text{ for }  i \neq j .
\end{alignat*}
\end{definition}

If $y \in Y_j$ and $z \in Y_{j+s}$ for $s \ge 1$ and there is a path in $\tr(r)$ connecting $y$ to $z$ then we say that $z$ is a {\it descendant}   of $y$ and that $y$ is an {\it ancestor} of $z$. If we let $y_0=z$, $y_s=y$, by path we mean  $P_{z,y} =\{ y_0, y_1, \ldots, y_s\}$ such that $y_i \in Y_{j+s-i}$ for $0 \le i \le s$ such that $e_i=\{  y_{i-1}, y_{i}\} \in  E_{j+s-(i-1)}$ for $1 \le i \le s$. We let $\D(y)$ denote the set of all descendants of $y$ and $\A(y)$ to be all the ancestors of $y$.  We know that if $z \in Y$ then there is a unique path $P_{z,r}$ (which we denote by $P_z$) from $z$ to the root $r$. We say that the tree $\tr(r)$ is {\it acceptable} if $P_{z}$ is acceptable for each $z \in Y$. Moreover, we define
\begin{equation}
\rho(\tr(r))= \max_{z \in Y} \rho(P_{z}),
\end{equation}
where $\rho(P_{z})$ is defined in \eqref{rhoP}.

We can now state the following result.

\begin{lemma}\label{treelemma}
Let $\tr(r)= Y \times E$ with $Y= \bigcup_{ 0 \le j \le M} Y_j$ and $E=\bigcup_{ 1 \le j \le M}  E_j $ be an acceptable tree of $\Th$ with root $r \in \Lh$.  Then, for any $p \in Q_h^2$ there exist $\vv \in \V_{h,0}^3$ such that  
\begin{subequations}\label{tree}
\begin{alignat}{1}
\text{ support}(\vv)\subset &\, \Omega_h(Y) \label{tree1} \\
(\dive \vv-p)(\sigma)&= 0 \quad \text{ for all } \sigma \in Y \label{tree2} \\
\dive \vv (\sigma) &= 0 \quad\text{ for all } \sigma\in \Sh \backslash Y .\label{tree3}
\end{alignat}
\end{subequations}
If in addition $\Th$ is quasi-uniform the following bound holds
\begin{equation}\label{tree4}
\|\nabla \vv\|_{L^2(\Omega_h(Y))} \le  C \Big(1+D_r\Big) \Big(1+\Upsilon(\tr(r)) \rho(\tr(r))\Big)\|p\|_{L^2(\Omega_h(Y))}, 
\end{equation}
where $\Upsilon(\tr(r))^2:= \max_{z \in Y} \left( \sum_{ y \in \A(z)} |\D(y)|\right)$. We recall that $D_r$ is given in \eqref{aux123}.
\end{lemma}

\begin{proof}
For every $z \in Y$ with $z \neq r$  there is a unique path $P_z \subset Y$ that connects $z$ to $r$. By Lemma \ref{lemmapath} we can find $\vv_z \in \V_{h, 0}^3$  such that 
\begin{subequations}\label{auxx}
\begin{alignat}{1}
\text{ support}(\vv_z)\subset  & \, \Omega_h(P_z), \label{auxx0} \\
(\dive \vv_z-p) (z)&= 0 \label{auxx1}\\
\dive \vv_z(\sigma)&=0 \quad \text{ for all } \sigma \in \Sh, \,\sigma \neq z, \,\sigma \neq r, \label{auxx2} \\
\|\nabla\vv_z\|_{L^\infty(\Omega_h(P_z))} \le  &C \, \rho(P_z) \max_{T \in \Th(z)} |p|_T(z)|,   \label{auxx3}
\end{alignat}
\end{subequations}
where $C$  only depends on the shape regularity of the mesh. 

Note that from \eqref{auxx3} and inverse estimates we get that for any $y \in \A(z)$, assuming $h_y \le C \, h_z$ (which holds if the mesh is quasi-uniform) we have
\begin{equation}\label{inq106}
\|\nabla \vv_z\|_{L^2(\Omega_h(y))} \le C  \rho(P_z) \|p\|_{L^2(\Omega_h(z))}.
\end{equation}
Then, we take 
\begin{equation*}
\ww=\sum_{z \in Y,  z \neq r} \vv_z . 
\end{equation*}
We then have the following properties of $\ww$
\begin{alignat}{2}
(\dive \ww-p) (y)&= 0 \quad && \text{ for all } y \in Y \backslash \{r\}\\
\text{ support}(\ww)\subset & \,\Omega_h(Y) \quad &&  \\
\dive \ww(\sigma)=&0, \quad && \sigma \notin Y. 
\end{alignat}
Since $\rr \in \Lh$ (using Definition \ref{localinterpolate})
 we can find $\rr \in \V_{h, 00}^3(r)$ so that 
\begin{equation}
(\dive \rr +\dive \ww-p) (r) = 0 
\end{equation}
\begin{equation}
\|\nabla \rr\|_{L^\infty(\Omega_h(r))} \le D_r \max_{T \in \mathcal{T}_h(r)} |(\dive \ww-p)|_T(r)| .
\end{equation}
In this case, using inverse estimates, we can show that 
\begin{equation}\label{auxr}
\|\nabla \rr\|_{L^2(\Omega_h(r))} \le C D_r (\|\nabla \ww\|_{L^2(\Omega_h(r))}+ \|p\|_{L^2(\Omega_h(r))}).
\end{equation}

We next set $\vv= \ww+\rr$ and we see that $\vv \in   \V_{h, 0}^3$ and \eqref{tree1}, \eqref{tree2} and \eqref{tree3} hold.

To get the bound \eqref{tree4} we assume that $\Th$ is quasi-uniform. Using the triangle inequality and \eqref{auxr}
\begin{equation*}
\begin{split}
\|\nabla \vv\|_{L^2(\Omega_h(Y))} 
&\le \|\nabla \ww\|_{L^2(\Omega_h(Y))}+ \|\nabla \rr\|_{L^2(\Omega_h(Y))} \\
&\le (1+C \, D_r) \big( \|\nabla \ww\|_{L^2(\Omega_h(Y))}+\|p\|_{L^2(\Omega_h(r))}\big).
\end{split}
\end{equation*}
Next, we estimate $\ww$:
\begin{alignat*}{2}
\|\nabla \ww\|_{L^2(\Omega_h(Y))}^2 \le & \sum_{y \in Y} \|\nabla \ww\|_{L^2(\Omega_h(y))}^2   && \\
\le &   \sum_{y \in Y}   \|\sum_{z \in \D(y)} \nabla \vv_z \|_{L^2(\Omega_h(y))}^2 && \text{ by } \eqref{auxx0} \\
\le &    \sum_{y \in Y}   (\sum_{z \in \D(y)} \| \nabla \vv_z \|_{L^2(\Omega_h(y))})^2 && \text{ by the triangle inequality }  \\
\le &   \sum_{y \in Y} |\D(y)| \sum_{z \in \D(y)}   \|\nabla \vv_z \|_{L^2(\Omega_h(y))}^2 && \text{ by H\"older's inequality }  \\
\le &  \, C  \rho(\tr(r))^2 \sum_{y \in Y} |\D(y)|  \sum_{z \in \D(y)}   \|p \|_{L^2(\Omega_h(z))}^2 && \text{ by } \eqref{inq106} \\
\le &  \, C  \rho(\tr(r))^2 \sum_{z \in Y}  \|p\|_{L^2(\Omega_h(z))}^2 \sum_{ y \in \A(z)} |\D(y)|    && \text{ interchange summation } \\
\le & 3 \, C \rho(\tr(r))^2 \Upsilon(\tr(r))^2 \|p\|_{L^2(\Omega_h(Y))}^2 && \text{ by definition of } \Upsilon(\tr(r)). 
\end{alignat*}
Taking square roots we get 
\begin{equation*}
\|\nabla \ww\|_{L^2(\Omega_h(Y))} \le C \, \rho(\tr(r))\, \Upsilon(\tr(r)) \|p\|_{L^2(\Omega_h(Y))}.
\end{equation*}
The result now follows.

\end{proof}

The next result follows immediately from the previous lemma.

\begin{theorem}\label{mainthm}
Suppose that we have $\{r_1, \ldots, r_t \} \in \Lh$ and corresponding acceptable trees\\ $\{\tr(r_1), \tr(r_2), \ldots, \tr(r_t) \}$. If $\tr(r_i)= Z_i \times E_i$ we require that  $\bigcup_{i=1}^t Z_i= \Sh$ and $Z_i \cap Z_j = \emptyset$  for $i \neq j$.  Then, for any $p \in Q_h^2$ there exist $\vv \in \V_{h,0}^3$ such that 
\begin{equation*}
(\dive \vv-p)(\sigma)= 0 \quad \text{ for all } \sigma \in \Sh 
\end{equation*}
and if $\Th$ is quasi-uniform 
\begin{equation*}
\|\nabla \vv\|_{L^2(\Omega)} \le  C (1+\bar{D})  (1+\bar{\Upsilon} \bar{\rho})\|p\|_{L^2(\Omega)}. 
\end{equation*}
where 
\begin{equation}
\bar{D}=  \max_{ 1\le i \le t} D_{r_i}, \quad \bar{\Upsilon} =  \max_{ 1\le i \le t} \Upsilon(\tr(r_i)),  \quad \bar{\rho} =  \max_{ 1\le i \le t} \rho(\tr(r_i)).
\end{equation}

\end{theorem}

Using Theorem \ref{mainthm} and Lemma \ref{lemma4} we have that the inf-sup condition \eqref{inf-sup} holds for $Q_h^2 \times \V_h^3$. 
 
\begin{corollary}\label{maincor}
Assuming the hypothesis Theorem \ref{mainthm} then the estimate \eqref{inf-sup} 
holds for $Q_h^2 \times \V_h^3$ with constants given by Lemma \ref{lemma4} 
and Theorem \ref{mainthm}.
\end{corollary}

From Theorem \ref{mainthm} and Lemma \ref{lemma4} we deduce that if the hypotheses
of Theorem \ref{mainthm} hold then $Q_h^2 \times \V_h^3$ is inf-sup stable  
(see Corollary \ref{maincor}) and in fact that $\dive \V_h^3=Q_h^2$. 
If we do not care about the inf-sup constant in \eqref{inf-sup} and only care 
if $\dive \V_h^3=Q_h^2$ then we can give weaker conditions. 
Inspecting the proof of Lemma \ref{treelemma} (and using Lemma \ref{lemma4}),
we can show the following.

\begin{theorem}\label{mildconditions}.
If for every $z \in \Sh$ there is exists an $r \in \Lh$ such that  there is an acceptable path $P_{z,r}$ between them, then $\dive \V_h^3=Q_h^2$.
\end{theorem}

\section{Relationship to Qin's result}
\label{sec:relateqin}

Results concerning the pair of spaces $V^k_h,\, Q^{k-1}_h$ were given 
in \cite{ref:QinThesis} for $k\leq 3$.
Here we review the case $k=3$.
For the case $k=1$, also see \cite{ref:QinZhangCrossedTriangles}, and 
for the case $k=2$, see \cite{arnold1992quadratic}.
Qin considered the mesh in Figure \ref{fig:fortyfive}, which is called a
Type I triangulation \cite{lai2007spline}.
Of course, the upper-left and lower-right triangles are problematic,
since the pressures will vanish at the corner vertices there.
But more interestingly, Qin found an additional spurious pressure mode
as indicated in Figure \ref{fig:spurious}(a).
We can relate this to the quantities $\decvar_i$ in \eqref{bigd} by computing them
for this mesh, as indicated in Figure \ref{fig:spurious}(b).
There are only two angles in this mesh, $\pi/4$ and $\pi/2$, and
$\cot(\pi/4)=1$ and $\cot(\pi/2)=0$.
Similarly, the edge lengths are $L$ and $L\sqrt{2}$, for some $L$.
Thus the quantities $d_{j0}$ in \eqref{Psiz} are of the form $\pm A$
where $A=1/2L^2$ for the $\pi/4$ angles, and 0 for the $\pi/2$ angles,
as indicated in Figure \ref{fig:spurious}(b).
Computing the alternating sum of terms in \eqref{Psiz}, we get 
\begin{equation}\label{eqn:deejayzed}
\decvar_0=\sum_{j=1}^6 (-1)^j d_{j0}=A-(-A)+0-A+(-A)=0.
\end{equation}
Thus condition \eqref{bigd} is violated for $i=0$ for all the interior vertices in 
Figure \ref{fig:fortyfive}.

Now let us compute $\decvar_i$  for $i=1,2$.  First, we note that the sequence of vertices $y_k$ for a fixed interior vertex $z$
\begin{alignat*}{1}
y_1=&L(1,0)+z,\;y_2=L(1,1)+z,\;y_3=L(0,1)+z, \\
y_4=&L(-1,0)+z,\; y_5=L(-1,-1)+z,\; y_6=L(0,-1)+z.
\end{alignat*}
Thus,
\begin{equation}\label{eqn:sumyytuu}
\sum_{j=1}^6  (-1)^j y_j=
L(-1+1-1+1,1-1+1-1)+z \sum_{j=1}^6 (-1)^j=\bfz.
\end{equation}
Letting $t_j=\cot(\theta_j)+\cot(\theta_{j+1})$, we easily can show
$$ t_1=1,\; t_2=2,\; t_3=1,\; t_4=1,\; t_5=2,\; t_6=1.  $$
Also, we have the sequence of values $|e_j|^{-2}$ are
$$
|e_1|^{-2}=\frac{1}{L^2} ,\;|e_2|^{-2}=\frac{1}{2L^2},\;|e_3|^{-2}=\frac{1}{L^2} ,\; 
|e_4|^{-2}= \frac{1}{L^2},\; |e_5|^{-2}=\frac{1}{2L^2},\; |e_6|^{-2}=\frac{1}{L^2}.
$$
Hence, $\frac{(\cot(\theta_j)+\cot(\theta_{j+1}))}{|e_j|^2}=\frac{1}{L^2}$ for all $j$. Also, $|T_j|=|T_1|$ for all $j$. Hence, using \eqref{eqn:alldeemone} , \eqref{eqn:deejayzed}, and \eqref{eqn:sumyytuu} we have for $i=1,2$ 
\begin{equation*}
\decvar_i=\left(\frac{4}{L^2}-\frac{2}{|T_1|}\right) \sum_{j=1}^n (-1)^j y_j \cdot \cde^\perp_i=0. 
\end{equation*}

Thus condtion \eqref{bigd} is violated for all $i=0,1,2$ for all interior vertices in 
Figure \ref{fig:fortyfive}. This suggests that the constraint \eqref{bigd} maybe required for inf-sup stability.

\begin{figure}
\centerline{(a)\includegraphics[width=1.4in]{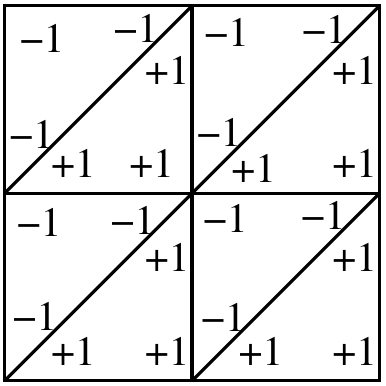}
\qquad\qquad (b) \includegraphics[width=1.4in]{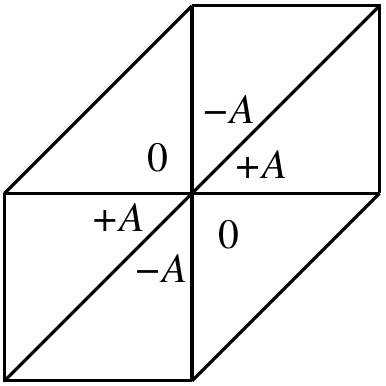}}
\caption{(a) A global spurious pressure mode on the mesh in Figure \ref{fig:fortyfive}.
(b) Computation of \eqref{eqn:deejayzed}.}
\label{fig:spurious}
\end{figure}

\newcommand{\Z}{{\bf Z}}

\section{Strang's Dimension}
\label{sec:strangdim}

For simplicity, let us assume that $\Omega$ is simply connected.
Then the space $\Z_h^3:= \{ \vv \in \V_h^3: \dive \vv=0 \}$ is the curl 
of the space $S^4_h$ of $C^1$ piecewise quartics on the same mesh, 
where the quartics must vanish to second order on the boundary.
The dimension of the space $\widehat S^4_h$ of $C^1$ piecewise quartics,
without boundary conditions, is known \cite{lrsBIBbv,alfeld1987explicit} 
for  a broad class of triangulations $\Th$ to be
\begin{equation} \label{eqn:strangdim}
\dim \widehat S^4_h=  E + 4V - V_0 + \sigma_i,
\end{equation}
where $T$ is the number of triangles in $\Th$, 
$E$ (resp., $E_0$) is the number of edges (resp., interior edges) in $\Th$, 
$V$ (resp., $V_0$) is the number of vertices (resp., interior vertices)
in $\Th$, and $\sigma_i$ is the number of singular interior vertices in $\Th$.
 
The dimension formula \eqref{eqn:strangdim} is essentially the one conjectured by
Gil Strang \cite{strang1973piecewise}, so we refer to this as 
the {\em Strang dimension} of $\widehat S^4_h$:
\begin{equation} \label{eqn:dimzeecomph}
\sdim(\widehat S^4_h)=E + 4V - V_0 + \sigma_i.
\end{equation}
For $C^1$ piecewise polynomials of degree $k\geq 5$, the Strang dimension 
was confirmed using an explicit basis \cite{lrsBIBaf}.
But the Strang dimension has not yet been confirmed for arbitrary meshes for $k\leq 4$.
However, what is known is that it provides a lower bound \cite{lrsBIBbv}
\begin{equation} \label{eqn:lowbnddim}
\dim \widehat S^4_h\geq \sdim(\widehat S^4_h).
\end{equation}

\subsection{Computing $\dim S^4_h$}

\begin{figure}
\centerline{(a)\quad\includegraphics[scale=.65]{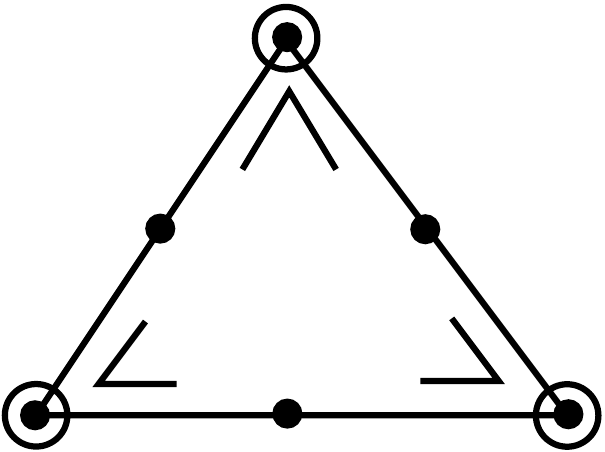}
   \qquad  \qquad   (b)\includegraphics[scale=.6]{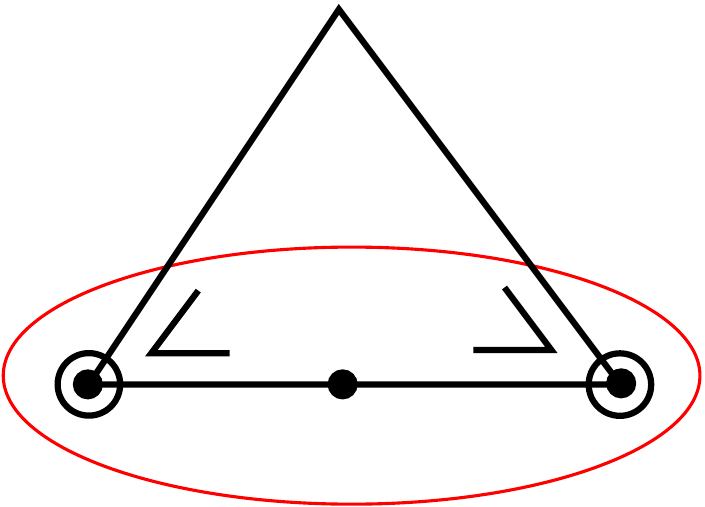}}
\caption{(a) Nodal basis for quartics.
(b) Nodes that determine value and gradient on an edge.}
\label{fig:quartics}
\end{figure}

Now let us compute $\dim S^4_h=\dim \Z_h^3$ under the assumption that the inf-sup
condition holds.
The space $\V^3_h$ can be described in terms of Lagrange nodes:
\begin{equation} \label{eqn:otherhahdnh}
\dim \V^3_h = 2 (T + 2 E_0 + V_0)  = 2 T + 4 E_0  +2 V_0 .
\end{equation}
We have $\sdiv \V^3_h\subset Q^2_h$, where the latter space consists of mean-zero 
piecewise quadratics that satisfy the alternating condition $A_h^z(\dive \vv)=0$
at singular vertices, where $A_h^z$ is defined in \eqref{eqn:ehaitchzeeq}:
$$
\dim Q^2_h= 6 T -1 - \sigma.
$$
Thus $\dim\sdiv \V^3_h=\dim Q^2_h-K$ for some integer $K\geq 0$, and
$\sdiv \V^3_h=Q^2_h$ if and only if $K=0$.
Since $$\dim \V^3_h=\dim(\hbox{image}\sdiv \V^3_h)+\dim \Z^3_h,$$ we have 
\begin{equation} \label{eqn:dimcompdith}
\begin{split}
\dim \Z^3_h &= \dim \V^3_h - \dim(\hbox{image}\sdiv \V^3_h) \\
&= 2 T + 4 E_0 + 2 V_0 - 6 T +1+\sigma +K\\
&= -4 T + 4 E_0 + 2 V_0 +1+\sigma +K.\\
\end{split}
\end{equation}
We have $3T=(E-E_0)+2E_0=E+E_0$. Thus
\begin{equation} \label{eqn:diddcompdith}
\dim \Z^3_h 
= - T -(E+E_0) + 4 E_0 + 2 V_0 +1+\sigma +K .\\
\end{equation}
By Euler's formula, $1=T-E+V=T-E_0+V_0=1$.
Thus $E_0-V_0=T-1$, and
\begin{equation} \label{eqn:biscompdith}
\begin{split}
\dim S^4_h = \dim \Z^3_h &= (V_0-E_0)-(E+E_0) + 4 E_0 + 2 V_0 +\sigma +K\\
&= 2 E_0 -E + 3 V_0 +\sigma +K.\\
\end{split}
\end{equation}
Technically, we actually have
$$
\dim S^4_h = \min\{0,2 E_0 - E+ 3 V_0 +\sigma+K\},
$$
since the dimension can never be negative.
There are cases where the number of boundary edges $E-E_0$ (which is the same
as the number of boundary vertices) is larger than $E_0 + 3 V_0 +\sigma$.
For a domain consisting of only two triangles, $E_0=1$, $V_0=0$, $\sigma=2$, and 
$E-E_0=4$, so the formula in \eqref{eqn:biscompdith} gives a negative number if $K=0$.
From now on, we assume that $E_0 -  (E-E_0) + 3 V_0 +\sigma\geq 0$ for $\Th$.

\begin{theorem}\label{thm:nearvstrang}
Let $k=3$ and suppose that $\Th$ is any triangulation satisfying
$E_0  + 3 V_0 +\sigma\geq  E-E_0$.
Then
$$ \sdiv \V^3_h=Q^{2}_h$$ if and only if
$$
\dim S^4_h = 2 E_0 - E + 3 V_0 +\sigma.
$$
More generally, 
$$ \dim \sdiv \V^3_h=\dim Q^{2}_h-K$$ if and only if
$$
\dim S^4_h = 2 E_0 - E + 3 V_0 +\sigma+K.
$$
\end{theorem}

\subsection{Computing $\dim \widehat S^4_h$}

Now let us relate the spaces $S^4_h$ and $\widehat S^4_h$ by
imposing boundary conditions on $\widehat S^4_h$ to yield the space $S^4_h$.
Using the approach pioneered by Strang \cite{strang1973piecewise}, it is natural to
conjecture that this involves simply imposing constraints on the boundary.
For example, a $C^1$ piecewise quartic that vanishes to second order on $\partial\Omega$
must vanish to second order at each boundary vertex (3 constraints per boundary vertex).
In addition, the value at one point on each boundary edge must vanish, as well as 
the normal derivative at two points on each boundary edge.

To see why this is the right number of constraints, we pick special nodal variables 
for quartics as indicated in Figure \ref{fig:quartics}(a).
These are 
\begin{enumerate}
\item the value and gradient at each vertex, 
\item the value at edge midpoints, and
\item the second-order cross derivatives $\partial_{e_i}\partial_{e_j}$ evaluated
at the vertex $\nu_{ij}$ at the intersection of $e_i$ and $e_j$,
where the $e_k$'s are the edges of the triangle.
\end{enumerate}
More precisely, $\partial_{e_i}\phi(\nu_{ij})$ is defined as the directional derivative
of $\phi$ in the direction of $e_i$ away from $\nu_{ij}$.
These nodal variables are unisolvent for quartics, as follows.
Vanishing of nodal variables of type (1) and (2) guarantee vanishing on each edge; 
these are the standard nodal variables for Hermite quartics.
Thus a quartic $q$ with these nodal values zero is of the form $q=L_1 L_2 L_3 L$ 
where the non-trivial linear functions $L_i$ vanish on $e_i$.
But 
$$
\partial_{e_i}\partial_{e_j}q(\nu_{ij})
=\big(\partial_{e_i}L_j\big) \big(\partial_{e_j}L_i\big)L_k(\nu_{ij}) L(\nu_{ij}),
$$
where $\{i,j,k\}=\{1,2,3\}$ and $L_k(\nu_{ij})\neq 0$.
Thus vanishing of the nodal variables of type (3) implies that $L\equiv 0$.

Moreover, similar arguments show that the nodal variables for $q$ associated 
with a boundary edge, as indicated in Figure \ref{fig:quartics}(b), 
determine $q$ to second order on that edge.
Thus satisfication of second-order boundary conditions is guarenteed by
setting these nodal values to zero.
This reduces the dimension by at most $3(V-V_0)+3(E-E_0)=6(E-E_0)$,
since $E-E_0=V-V_0$. 
We will show that these conditions have (at least) one redundancy at singular 
boundary vertices.
Thus
$$
\dim S^4_h \geq \dim \widehat S^4_h -6(E-E_0)+\sigma_b,
$$
where $\sigma_b$ is the number of singular boundary vertices.
Therefore
\begin{equation} \label{eqn:predimzeecomph}
\dim S^4_h +6(E-E_0)-\sigma_b \geq \dim \widehat S^4_h \geq \sdim(\widehat S^4_h)=
E + 4V - V_0 + \sigma_i.
\end{equation}
Assume now that $\sdiv \V^3_h=Q^{2}_h$.
Using Theorem \ref{thm:nearvstrang} and $E-E_0=V-V_0$, we find
\begin{equation} \label{eqn:pexpnazeecomph}
\begin{split}
\dim S^4_h &+6(E-E_0)-\sigma_b =
 E_0 -  (E-E_0) + 3 V_0 +\sigma +6(E-E_0)-\sigma_b \\
 &= E_0 + 3 V_0 +\sigma_i +5(E-E_0) = E_0 + 3 V +\sigma_i +2(E-E_0) \\
 &= E + 3 V +\sigma_i +(E-E_0) = E + 4 V +\sigma_i -V_0 .\\
\end{split}
\end{equation}
Combining \eqref{eqn:predimzeecomph} and \eqref{eqn:pexpnazeecomph}
proves the following result.

\begin{theorem}\label{thm:equivstrang}
Let $k=3$ and suppose that $\Th$ is a triangulation satisfying
$E_0  + 3 V_0 +\sigma\geq  E-E_0$ and suppose that $\sdiv \V^3_h=Q^{2}_h$ on this triangulation. Then, the Strang dimension
\eqref{eqn:strangdim}  is valid for $\widehat S^4_h$, $\dim \widehat S^4_h= \sdim(\widehat S^4_h)$.
Moreover, equality holds in \eqref{eqn:predimzeecomph}, so the 
$6(E-E_0)-\sigma_b$ constraints are nonredundant.
\end{theorem}

\begin{figure}
\centerline{(a)\includegraphics[scale=.65]{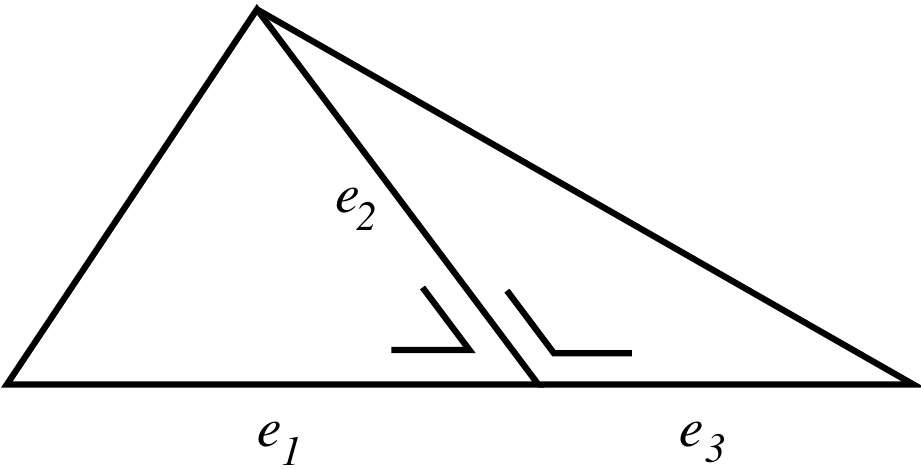}
  \qquad    (b)\includegraphics[scale=.6]{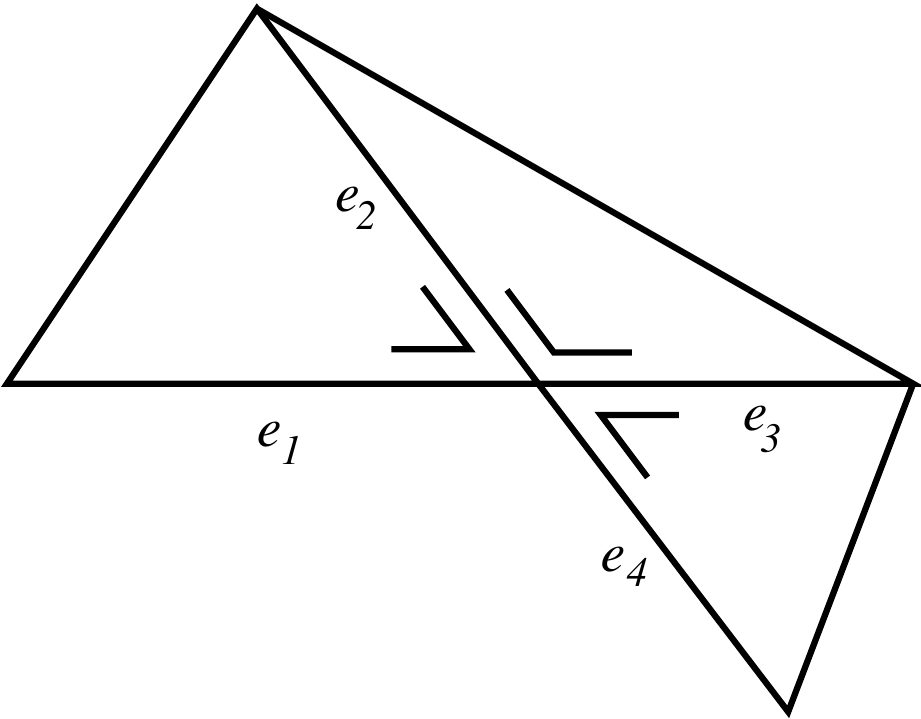}}
\caption{(a) Singular boundary vertex where two triangles meet.
(b) Singular boundary vertex where three triangles meet.}
\label{fig:singbdry}
\end{figure}

To complete the proof of Theorem \ref{thm:equivstrang}, we need to verify
the redundancy of constraints at singular boundary vertices.
This occurs because the 
second-order cross derivatives $\partial_{e_i}\partial_{e_j}$ are linearly
dependent at singular boundary vertices.
For the case of a triangle with two boundary edges $e_1$ and $e_2$,
the vanishing of the nodal variables of type (1) and (2) on $e_1$ and $e_2$
already imply vanishing on both edges, so necessarily 
$\partial_{e_1} \partial_{e_2}$ is already zero.

For the case where two triangles meet at a singular boundary vertex $\nu$, 
see Figure \ref{fig:singbdry}(a).
Then $e_1$ and $e_3$ are parallel, and thus
\begin{equation} \label{eqn:singbdrya}
\partial_{e_1} \partial_{e_2} \phi(\nu) =-\partial_{e_3} \partial_{e_2}\phi(\nu) 
\end{equation}
for any $C^1$ piecewise quartic $\phi$.
Thus setting one of them to zero sets the other; they are redundant.

For the case where three triangles meet at a singular boundary vertex $\nu$, 
see Figure \ref{fig:singbdry}(b).
Equation \eqref{eqn:singbdrya} still holds, and in addition
$e_2$ and $e_4$ are parallel, and thus
$$
\partial_{e_3} \partial_{e_2} \phi(\nu) =-\partial_{e_3} \partial_{e_4}\phi(\nu) 
$$
for any $C^1$ piecewise quartic $\phi$.
Thus
$$
\partial_{e_1} \partial_{e_2} \phi(\nu) =\partial_{e_3} \partial_{e_4}\phi(\nu) ,
$$
and setting one of them to zero sets the other; they are redundant.
This completes the proof of Theorem \ref{thm:equivstrang}.

Since Theorem \ref{mainthm} allows us to prove the inf-sup condition 
for quite general meshes, this confirms the Strang dimension 
\eqref{eqn:strangdim} for $\widehat S^4_h$ is correct for such meshes. In particular, if the hypothesis of Theorem \ref{mildconditions} holds then \eqref{eqn:strangdim} is correct. 

\subsection{Connection to Qin's results}

There is a connection between Qin's results and dimension counting.
Qin finds a spurious mode that suggests that $\dive \V^3_h\neq Q^2_h$ on
the right-traingle mesh in Figure \ref{fig:fortyfive}.
We conclude that the dimension of the space $S_h^4$ of $C^1$ quartics satisfying 
second-order boundary conditions on this mesh is at least one larger than the 
dimension for this space given in Theorem \ref{thm:nearvstrang}: $K\geq 1$.
On the other hand, it is well known \cite{lai2007spline,lrsBIBbv} that the
Strang dimension \eqref{eqn:strangdim} is correct on Type I triangulations 
without boundary conditions.
In view of \eqref{eqn:predimzeecomph}, there is a further redundancy in 
the constraints enforcing boundary conditions.
Unfortunately, the dimension of splines in two dimensions satisfying boundary conditions 
has had only limited study \cite{typeIItrianglebcs,anbo1989dimension} so far.

\begin{figure}
\centerline{\includegraphics[width=2.7in]{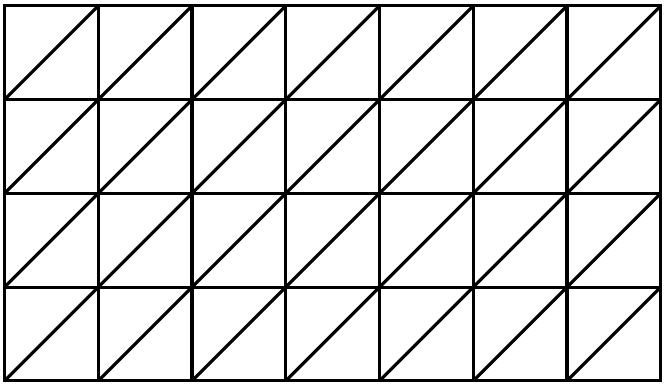}}
\caption{The Type I regular mesh studied in \cite[Chapter 6]{ref:QinThesis}}
\label{fig:fortyfive}
\end{figure}

\bibliography{Bibliography_BGSS}{}
\bibliographystyle{abbrv}

\end{document}